\documentclass[english]{amsart}

\usepackage{amssymb}
\usepackage{amsthm}
\usepackage{amsmath}
\usepackage[all]{xy}
\usepackage{color}
\usepackage{mosaicosPBC}
\usepackage{subfigure}
\usepackage{float}
\xyoption{arc}
\usepackage{microtype}
\usepackage{accents}

\newcommand{\red}{\color{red}}

\newcommand{\tq}{\mid}
\newcommand{\Z}{\mathbb{Z}}
\newcommand{\R}{\mathbb{R}}

\newcommand{\N}{\mathbb{N}}
\newcommand{\T}{\mathit{\mathbb{T}}}
\newcommand{\X}{\mathbb{X}}
\newcommand{\Pa}{\mathbb{P}}
\newcommand{\Q}{\mathbb{Q}}
\newcommand{\B}{\mathbb{B}}
\newcommand{\M}{\mathbb{M}}
\newcommand{\calc}{\mathcal{C}}
\newcommand{\cald}{\mathcal{D}}
\newcommand{\calp}{\mathcal{P}}
\newcommand{\calb}{\mathcal{B}}
\newcommand{\calt}{\mathcal{T}}
\newcommand{\calk}{\mathcal{K}}

\newcommand{\calr}{\mathcal{R}}
\newcommand{\calf}{\mathcal{F}}
\newcommand{\calh}{\mathcal{H}}
\newcommand{\pp}{\Pa'}
\newcommand{\ppp}{\Pa''}
\newcommand{\pb}{\calb'}
\newcommand{\ppb}{\calb''}
\newcommand{\cpartial}{\check{\partial}}
\newcommand{\scirc}{{\scriptstyle \circ}}
\newcommand{\norm}[1]{\lVert #1 \rVert}

\newcommand{\inte}[1]{\mathaccent23{#1}}
\newcommand{\card}{\#}      %Simbolos
\DeclareMathOperator{\area}{area}  
\DeclareMathOperator{\longi}{length}  

\newlength{\dhatheight}

\newtheorem{theorem}{Theorem}[section]

\newtheorem{lemma}[theorem]{Lemma}
\newtheorem{proposition}[theorem]{Proposition}

\theoremstyle{definition}
 \newtheorem{definition}[theorem]{Definition}

\newtheorem{remark}[theorem]{Remark}

\begin{document}

\title[A geometric proof of the Affability Theorem]{A geometric proof of the Affability Theorem for planar tilings}

%1st author
\author[F. Alcalde]{Fernando Alcalde Cuesta} 
\address{Departamento de Xeometr\'{\i}a e Topolox\'{\i}a \\ Facultade de Matem\'aticas \\
Universidade de Santiago de Compostela \\ R\'ua Lope G\'omez de Marzoa s/n \\ 
E-15782 Santiago de Compostela (Spain)}
\email{fernando.alcalde@usc.es}

%2nd author
\author[P. G. Sequeiros]{Pablo Gonz\'alez Sequeiros}
\address{Departamento de Did\'actica das Ciencias Experimentais \\ Facultade de Formaci\'on do Profesorado \\ 
Universidade de Santiago de Compostela \\  Avda. Ram\'on Ferreiro, 10 \\ E-27002 Lugo (Spain)}
\email{pablo.gonzalez.sequeiros@usc.es}

%3th author
\author[\'A. Lozano]{\'Alvaro Lozano Rojo}
\address{Centro Universitario de la Defensa - IUMA Universidad de Zaragoza \\
Academia General Militar \\ Ctra. Huesca s/n \\ E-50090 Zaragoza (Spain)}
\email{alvarolozano@unizar.es}

\date{\today}
\thanks{Partially supported by the Ministry of Science and Innovation - Government of Spain (Grant 
\\  \hspace*{1em} MTM2010-15471), the University of the Basque Country (Grant EHU09/04) and the Xunta de \\  \hspace*{1em}  Galicia (IEMath Network CN 2012/077).}



\begin{abstract}
We give a geometric proof of the Affability Theorem of  T. Giordano, H. Matui, I. Putnam and  C. Skau for aperiodic and repetitive planar tilings. 
\end{abstract}

\maketitle

%%%%%%%%%%%%%%%%%%%%%%%%%%%%%%%%%

\section{Introduction}
\label{introduccion}

In this paper, we give a more accessible proof of a deep theorem by T. Giordano, H. Matui, I. Putnam and  C. Skau \cite{GMPS1} on the orbit structure of minimal dynamical systems on the Cantor set. We combine the absorption techniques from \cite{GMPS2} with some new ideas and techniques, which are inspired by the famous tiling constructed by R. M. Robinson \cite{Ro} to disprove the Wang conjecture about the decidability of the Tiling Problem \cite{GS}. 
\medskip 

For us, a dynamical system is an action of a countable group of transformations, or more generally an \'etale equivalence relation (EER). Recall that 
an equivalence relation $\calr$ on a second-countable locally compact Hausdorff space $X$ is said to be {\em \'etale} if $\calr$ admits a topology that makes it a locally compact Hausdorff $r$-discrete groupoid, so the projection maps $r,s : \calr \to X$ are local homeomorphisms. The orbit equivalence relation $\calr = \{ \, (x,g.x) \in X \times X \, | \, x \in X , g \in G \, \}$ defined by a countable discrete group $G$ acting continuously on $X$ is the basic example of EER. 
\medskip 
%becomes an EER by transferring the product topology on $G \times X$ to $\calr$. 
%Given a free left action $(g,x) \in G \times X \mapsto g.x \in X$ of a countable discrete group $G$ on $X$,  the orbit equivalence relation $\calr = \{ \, (x,g.x) \in X \times X \, / \, x \in X , g \in G \, \}$ becomes an EER when $\calr$ is endowed with the topology obtained from the product topology on $G \times X$.

From a dynamical point of view, a main problem is to determine when two EERs $\calr$ on $X$ and $\calr'$ on $X'$ are orbit equivalent. As for group actions, we say $\calr$ and  $\calr'$ are {\em orbit equivalent } (OE) if there is a homeomorphism \mbox{$\varphi : X \to X'$} such that 
\mbox{$(\varphi \times \varphi)(\calr) = \calr'$}. 
In measurable dynamics, the study of orbit equivalence was initia\-ted by H. A. Dye \cite{D1,D2} for group actions.
%in \cite{D1} and \cite{D2}. 
%the late 1950s. 
Pursued by W. Krieger \cite{K}, J. Feldman and D. A. Lind \cite{FL}, and 
D. S. Ornstein and B. Weiss \cite{OW1} among others authors, the idea consisted of finding a Rohlin approximation (similar to the approximation by periodic transformations of any aperiodic nonsingular transformation of a probability space) for the group. An analogous method was used by C. Series \cite{Ser} to show that any measurable equivalence relation $\calr$ with polynomial growth is {\em hyperfinite}. This means that  $\calr$ is the increasing union of countably many finite  equivalence relations on a full measure set, or equivalently $\calr$ is defined by a measurable $\Z$-action. 
%This result was extended by M. Samuelid\`es to the quasi-invariant case. 
The equivalence between hyperfiniteness and amenability was finally proved by A. Connes, J. Feldman and B. Weiss 
in their celebrated paper \cite{CFW}.
%a celebrated paper of 1981.
%, and revisited later by V. A. Kaimanovich. As a corollary, all ergodic amenable equivalence relations are measurewise orbit equivalent. 
\medskip 

In the topological setting, the solution to the same problem is more subtle. Firstly, ergodicity for measurable equivalence relations is replaced by minimality of EERs, which means that all equivalence classes are dense. Moreover, since any orbit equivalence reduces to an isomorphism for connected spaces, we can focus on minimal dynamical systems defined on totally disconnected spaces.  According to a strategy drawn by Giordano, Putnam and Skau in a series of papers \cite{GPS0,GPS} previous to the aforementioned \cite{GMPS1}, the idea consists of approaching minimal EERs by finite equivalence relations 
%(up to orbit equivalence) 
and providing an invariant of orbit equivalence being complete for these approximate finite (AF) equivalence relations. Following \cite{GPS}, an EER $\calr$ on $X$ is said to be {\em affable}  if $\calr$ is OE to an AF equivalence relation. Since the classification was completed in \cite{GPS0}, the main problem remains to prove that any minimal amenable EER on a totally disconnected compact space is affable \cite{GPS}. For the amenability of \'etale groupoids and minimal laminations, see \cite{AnR} and \cite{AR} respectively.
\medskip 

The case of $\Z$-actions was studied in \cite{GPS}. In this paper, the authors show any AF equivalence relation on the Cantor set is represented by a combinatorial object, called a {\em Bratteli diagram}, providing an orbit equivalence with a Cantor minimal $\Z$-system. Reciprocally, any EER arising from a Cantor minimal $\Z$-system is affable. To prove this result, they introduce the idea of a \lq small\rq~extension of a minimal AF equivalence relation, which is obtained by a sort of Rohlin approximation and described in a pure combinatorial way. The same ideas and techniques was used in the remarkable papers \cite{GMPS1} and \cite{GMPS3} to prove that minimal free $\Z^2$-actions and 
$\Z^m$-actions on the Cantor set are OE to 
minimal free $\Z$-actions. Nevertheless, the first step of the proof also involves subtle geometrical and combinatorial arguments in both cases $m=2$ and $m >2$. These %geometrical and combinatorial 
methods are combined with an important result regarding the extension of minimal AF equivalence relations, called the Absorption Theorem \cite{GMPS2}, whose first version was stated and proved 
in \cite{GPS}.
%and used in \cite{M1}.
\medskip

Cantor dynamical systems appear naturally in the study of laminations defined by tilings and graphs \cite{G}. In fact, the geometry of tilings plays a important role in the proof of the Affability Theorems in \cite{GMPS1} and \cite{GMPS3}. Any repetitive planar tiling satisfying a finiteness pattern condition 
%, which is called {\em Finite Pattern Condition} in \cite{BBG} 
defines an interesting minimal dynamical system on a compact connected space \cite{BBG,G}, which is said to be a {\em tiling dynamical system} (TDS).  If we choose a point on each tile, and we force the tiling to have the origin on one of these points,  we obtain a  
totally disconnected closed subspace $X$. Such a subspace meets all the orbits, and the TDS induces a minimal EER $\calr$ on $X$.  
For repetitive planar tilings, $\calr$ is OE to a minimal free $\Z^2$-action. Reciprocally, using the classical suspension method \cite{GMPS2}, the orbit equivalent relation defined by any Cantor minimal $\Z^2$-system can be realized as the EER induced by a minimal dynamical 
$\R^2$-system.
 %whose orbits are tilable. 
However, when we restrict $\calr$ to any clopen subset,  we need that each induced class to be realized as base point (Delone) set of 
some Voronoi tiling. 
%Extending an earlier result of  Matui \cite{M1} for a class of substitution tilings, 
In fact, the proof of  the Affability Theorem for $m=2$ involves a fine control on the geometry of Voronoi tiles  \cite{GMPS1}. For dimension $m >2$, the same authors have had to modify their geometrical arguments on Voronoi tilings, as well as the combinatorial ones related with the Rohlin approximation \cite{GMPS3}. 
\medskip

Our purpose in this paper is to simplify this part of the proof of the Affability Theorem for aperiodic and repetitive planar tilings. We focus our attention on the $2$-dimensional case to facilitate intuition and to make more accessible the proof, but the first advantage of our approach is that there is no fundamental difference between the cases $m=2$ and $m >2$.  Moreover, since Voronoi tiling are not involved in our proof, it might make easier to extend it to aperiodic tilings for other amenable Lie groups as the Heisenberg group $H^3$ or the solvable group $Sol^3$. We believe that any Cantor dynamical system arising from an aperiodic and repetitive tiling of the nilpotent group $H^3$ is affable, since there is no significant differences with the abelian case $m =3$. But we do not know the answer for $Sol^3$ which admit a solvable cocompact discrete subgroup with exponential growth.
%However, the answer to the main problem as settled above is not always positive: although the affine group is amenable, any Cantor dynamical system arising from an affine action is not affable because it does not admit transverse invariant measures \cite{P}.
\medskip

One of the first and most remarkable examples in the theory of aperiodic tilings was constructed by R. M. Robinson \cite{Ro} from a set of $32$ aperiodic prototiles, $6$ up to isometries of the plane. For any repetitive Robinson tiling, there is a Borel isomorphism from the space of sequences of $0$s and $1$s equipped with the cofinal equivalence relation to a full measure and residual set of orbits of the corresponding TDS. Since all cofinality classes are orbits of a $2$-adic adding machine, except for one orbit that is the union of two cofinality classes, this Cantor minimal dynamical system is measurewise OE to a $2$-adic adding machine. By construction, any repetitive Robinson tiling is the increasing union of patches which are inflated from patches at the previous step. So this inflation process gives us a finite approximation of the corresponding orbit equivalence relation $\calr$ by a minimal open AF equivalence subrelation $\calr_\infty$. In this example, it is quite easy to see how \lq small\rq~is the difference between $\calr$ and $\calr_\infty$. We can actually identify the boundary set $\partial \calr_\infty$, made up of the points in the boundary of the $\calr_\infty$-classes into the $\calr$-classes. Then we can deduce that $\partial \calr_\infty$ is indeed \lq small\rq~(in the sense that is specified in \cite{GPS}) from a growth argument. By using the Absorption Theorem \cite{GMPS2},  we can obtain an orbit equivalence with a Cantor minimal $\Z$-system. 
\medskip 

\begin{figure}
\includegraphics[width=50mm]{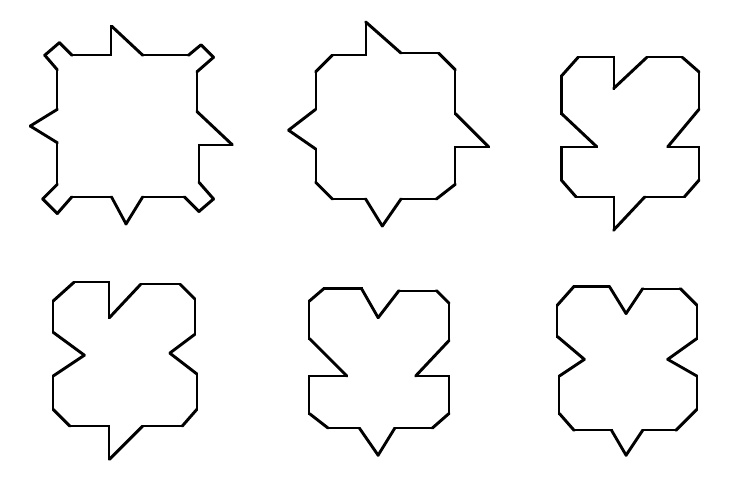}
\caption{\label{fig:arbolescantor} Robinson's aperiodic tiles}
\end{figure}

In this paper, we make use of Robinson's example to define a special inflation process, which we call {\em Robinson inflation}, verifying most of the properties needed to apply absorption techniques from \cite {GMPS2}. In \cite{AGL}, we announced a proof of the Affability Theorem where we used a former version of the Absorption Theorem. Unfortunately, there are two gaps in the proof of Lemma 3.1 and Theorem 5.1, which we close now by defining this new inflation process. Anyway, we shall keep the general schema given in \cite{AGL}. Thus, denoting by 
$\calr$ the EER induced by a TDS on any closed subset $X$ determined by the choice of a base point in each tile, we distinguish three  steps:

\begin{list}{--}{\leftmargin=1em}

	\item The first step consist of applying the \emph{inflation} or \emph{zooming process} of \cite{BBG}  to obtain an open AF equivalence subrelation $\calr_\infty$ of $\calr$. 

	\item In the second step, we introduce the \emph{discrete boundary} $\partial \calr_\infty$ of $\calr_\infty$, as well as its {\em continuous} or 
		geometrical counterpart $\partial _c \calr_\infty$.

	\item In the third step, we review all the conditions needed to apply the Absorption Theorem \cite {GMPS2}. 

\end{list}
These general steps are grouped in the first part of the paper developed in Section~\ref{schema}. In Section~\ref{buenainflacion}, we describe Robinson inflation, allowing us to deal with the third part in Sections~\ref{sboundary},~\ref{filtrating}~and~\ref{absorbing}. We start by reducing planar tilings to square tilings according to a theorem by  L.~Sadun and R.~F.~Williams \cite{SW}. At each iteration of the inflation process, the new inflated tiling is still of finite type, although the prototiles become more and more complicated and highly nonconvex (as a complicated version of the Amman-Penrose tiles  described in \cite{G} and \cite{GS}). But they look as squares on a large scale, so the isoperimetric ratios of any nested sequence of tiles converge to $0$.
% in Proposition~\ref{isoperimetric2}. 
At the beginning of Section~\ref{sboundary}, using this fundamental result, we prove that the whole boundary is \lq small\rq. In this Section~\ref{sboundary}, we show another important property of the boundary. As before, when we consider any planar orbit of the TDS, it is very complicate to know the real look of the trace of the continuous boundary $\partial _c \calr_\infty$ (which splits the orbit into several regions corresponding to $\calr_\infty$-classes). But once again, by looking on a large scale, we can see that every $\calr$-equivalence class separates into at most four $\calr_\infty$-equivalence classes. However, since the boundary is not $\calr_\infty$-\'etale, it does not suffice to apply the Absorption Theorem of \cite{GMPS2}. 
%As we will show in Section~\ref{filtrating}, there is a fundamental domain for the equivalence relation $\calr$ in restriction to a portion of the boundary, but it does not allow us to conclude. 
 This problem is solved in a somewhat different way than in \cite{GMPS1} by coloring the tiles of all the inflated tilings. Notice, however, that this trick (which we develop in Section~\ref{filtrating} to make \'etale the original boundary) has the similar effet as that applied in \cite{GMPS1}, although it is not always possible to obtain an OE between $\calr$ and $\calr_\infty$.
%Nonetheless, this trick (which we use in Section~\ref{filtrating} to make \'etale the original boundary) is actually equivalent to that used in \cite{GMPS1}. Using this idea, the absorption of the boundary (developed in the last section Section~\ref{absorbing}) becomes similar to the original one in the previously mentioned paper. 
Finally, using this idea, the absorption of the boundary is finally accomplished in Section 7. 

\medskip 
 
To conclude,
%although we have seen that the answer to the main problem is not always positive, 
we hope that our method helps to achieve the goal stated in \cite{GPS} 
(where the problem is formulated using group actions instead equivalence relations) by proving that the orbit equivalence relation defined by a Cantor minimal dynamical system with subexponential growth is affable. More specifically, we think that Series' ideas may be extremely useful to solve the polynomial growth case.
 
\section{Tiling dynamical systems and  \'etale equivalence relations}
\label{afables}

In this section, we present the two basic notions of the paper, {\em tiling dynamical systems} (TDS) and {\em \'etale equivalence relations} (EER), although frequently we do not distinguish both concepts. 
\medskip 

A \emph{planar tiling} is a partition of $\R^2$ into polygons touching face-to-face, called \emph{tiles}, obtained by translation from a finite set of \emph{prototiles}. Such a tiling always satisfies the \emph{Finite Pattern Condition} given in \cite{BBG}. It is said to be \emph{aperiodic} if it has no translation symmetries, and  \emph{repetitive} if for any patch $M$, there exists a constant $R  > 0$ (depending only on the diameter of $M$) such that any ball of radius $R$ contains a translated copy of $M$. 
\medskip 

Let $\T(\calp)$ be the set of tilings $\mathcal T$ obtained from a finite set of prototiles $\mathcal P$. It is possible to endow $\T(\calp)$ with the \emph{Gromov-Hausdorff topology} \cite{BBG,G}  generated by the basic neighborhoods
\[
	U^r_{\varepsilon,\varepsilon'} =
		\{\, \calt'  \in \T(\calp)  \tq
		\exists \, v,v' \in \R^2 : \norm{v} < \varepsilon , \norm{v'} < \varepsilon', R(\calt+v,\calt'+v') > r \,\},
\]
where $R(\calt,\calt')$ is the supremum of radii $R>0$ such that $\calt$ and $\calt' $ coincide on the ball $B(0,R)$. Then $\T(\calp)$ becomes a compact metrizable space, which is naturally laminated by the orbits $L_\calt$ of the natural ${\mathbb R}^2$-action  by translation. 
For each $\calt \in \T(\calp)$, we denote by $\cald_\calt$ the Delone set determined by the choice of base points in the prototiles. Now $T  = \{\, \calt \in \T(\calp) \tq 0 \in \cald_{\calt} \,\}$ is a totally disconnected closed subspace which meets all the leaves, so $T$ is a \emph{total transversal} for $\T(\calp)$. 
\medskip 

If $\calt \in \T(\calp)$ is a repetitive tiling, then $\X = \overline{L}_\calt$ is a minimal closed subset of $\T(\calp)$, called the \emph{continuous hull of $\calt$}. If $\calt$ is also aperiodic, any tiling in $\X$ has the same property  and hence $X = T \cap \X$ is homeomorphic to the Cantor set. Then $\calr = \{\, (\calt,\calt-v) \in X \times X \tq v \in \cald_\calt \,\}$  is an EER on $X$, which completely represents the transverse dynamics of the TDS.. 
\medskip

Let $\calr$ be any EER on a second countable locally compact Hausdorff space $X$. Following \cite{GPS}, we say
that $\calr$ is a \emph{compact \'etale equivalence relation} (CEER) if
$\calr - \Delta_X$ is a compact subset of $X \times X$ (where $\Delta_X$ is the
diagonal of $X \times X$). This means that $\calr$ is \emph{proper} in the sense of \cite{Re}
and  trivial out of a compact set.

\begin{definition}[\cite{GPS}]
	An equivalence relation $\calr$ on a totally disconnected space $X$ is \emph{affable}
	if there exists an increasing sequence of CEERs $\calr_n$ such that
	$\calr=\bigcup_{n\in\N}\calr_n$.  The inductive limit topology turns $\calr$ into an
	EER and we say that $\calr=\varinjlim \calr_n$ is \emph{approximately finite} (or AF). 
\end{definition}

An example of AF equivalence relation is the cofinal equivalence relation on
the infinite path space of certain type of oriented graphs $(V,E)$, called
\emph{Bratteli diagrams}.  According to \cite{GPS0} and \cite{GPS}, their vertices are stacked on levels and their edges join two consecutive levels. More precisely, we denote by $V_n$ the set of vertices of the level $n$ and by $E_n$ the set of edges $e$ with origin $s(e)$ in $V_{n-1}$ 
and endpoint  $r(e)$ in $V_n$ in such a way that $V = \bigcup V_n$ and $E = \bigcup E_n$. In fact, they are actually the only examples of AF equivalence
relations:

\begin{theorem}[\cite{GPS}, \cite{Re}] 
	Let $\calr$ be an AF equivalence relation on a totally disconnected space $X$. There
	exists a Bratteli diagram $(V,E)$ such that $\calr$ is isomorphic to the tail
	equivalence relation on the infinite path space 
	\[
		X_{(V,E)} =
			\{\, (e_1,e_2,\dots) \tq e_i \in E_i, r(e_i) = s(e_{i+1}), \forall i \geq 1 \,\}
	\]
	given by 
	\[
		\calr_{cof} =
			\{\, ((e_1,e_2,\dots), (e'_1,e'_2,\dots)) \in X_{(V,E)} \times X_{(V,E)} \tq
														\exists \, m \geq 1  :  e_n = e'_n , \forall \, n \geq m \,\}.
	\]
	If $X$ is compact, then $(V,E)$ can be chosen \emph{standard}, i.e. $V_0 = \{v_0\}$ and
	$r^{-1}(v) \neq \emptyset$ for all $v \in V-\{v_0\}$. Furthermore, $\calr$ is minimal
	if and only if $(V,E)$ is \emph{simple}, i.e. for each $v \in V$, there is $m \geq 1$ such that all vertices in $V_m$ are reachable from $v$.  \qed
	
	\end{theorem}

As for the continuous hull of a Robinson repetitive tiling, the cofinal
equivalence relation on the infinite path space of a simple \emph{ordered} Bratteli
diagram (i.e. having a linear order on each set of edges with the same endpoint) is
essentially isomorphic to a Cantor minimal $\Z$-system. Indeed, by using
lexicographic order on cofinal infinite paths and sending the unique maximal path to
the unique minimal path, we have a minimal homeomorphism 
$\lambda_{(V,E)} : X_{(V,E)} \to  X_{(V,E)}$, called a \emph{Vershik map}, see \cite{GPS0, GPS}. This map preserves cofinality, except for the maximal and minimal paths. We 
refer to the corresponding dynamical system as the \emph{Bratteli-Vershik $\Z$-system}
associated to $(V,E)$. Now,  as proved in \cite{GPS}, it is not difficult to see that any minimal AF equivalence relation on the Cantor set $X$ is OE to its Bratteli-Vershik $\Z$-system.

%%%%%%%%%%%%%%%%%%%%%%%%%%%%%%%%%
\section{Theorem statement and proof schema}
\label{schema} 

In this section, we  describe the general schema of the proof of the main result: 

\begin{theorem}[{\bf Affability Theorem}, \cite{GMPS1}]
	\label{affabilitytheorem}
	The continuous hull of any aperiodic and repetitive planar tiling is affable. 
\end{theorem}

\noindent
As explained above, we distinguish three steps: 
\begin{list}{\labelitemi}{\leftmargin=1.5em}
	\item[1)] The first step consist of applying the \emph{inflation} or \emph{zooming process}
		developed in \cite{BBG}  to obtain an increasing sequence of CEERs $\calr_n$, 
		and thus an open AF equivalence subrelation $\calr_\infty= \varinjlim \calr_n$ of $\calr$. 

	\item[2)] In the second step, we define a \emph{discrete boundary $\partial \calr_\infty$
	          of $\calr_\infty$} and we study its properties. It is a nonempty meager closed subset of 
	          $X$ whose saturation contains all the points with $\calr_\infty$-equivalence class 
	          different from its $\calr$-equivalence class.

	\item[3)] In the third step, to apply the \emph{Absorption Theorem} of \cite{GMPS2}, we must distinguish three other steps: 
		
		\begin{list}{--}{\leftmargin=1em}
		
		\item Firstly, the discrete boundary $\partial \calr_\infty$ Section
		must be \emph{$\calr_\infty$-thin} in the sense of \cite{GPS}. This means that 
		$\mu(\partial\calr_\infty)=0$ for every $\calr_\infty$-invariant  probability 
		measure $\mu$. It is also important to show that all $\calr$-equivalence classes 
		split into a (uniformly bounded) finite number of $\calr_\infty$-equivalence classes. In Section~\ref{buenainflacion}, 
		we present a special inflation process, which we call \emph{Robinson inflation}, allowing us to construct a sequence of 
		transverse CEERs with these properties (which will be proved in Section~\ref{sboundary}). However, in order to apply absorption 
		techniques from \cite{GMPS2}, we also need to see that $\partial \calr_\infty$ is \emph{$\calr_\infty$-\'etale} (i.e. 
		the equivalence relation induced by $\calr_\infty$ on $\partial \calr_\infty$ is \'etale), but it is not true. 
		\item In the second step, 
		which occupies the whole Section~\ref{filtrating}, we replace $\calr_\infty$ with a minimal open AF subrelation 
		$\hat{\calr}_\infty$ such that $\partial \calr_\infty$ is $\hat{\calr}_\infty$-\'etale. 
		
		\item
		In the last step, corresponding to 
		Section~\ref{absorbing}, we finally apply the required techniques for the boundary absorption. 
		\end{list}
\end{list}
%Although the main idea of these two final sections is different from that used in \cite{GMPS1}, they play a similar role.
\subsection{Inflation} 
\label{inflacion} 

We start by recalling the general \emph{inflation process}
developed in \cite{BBG}. By definition of its topology, the continuous hull
$\X$ admits a \emph{box decomposition} $\calb=\{ \B_i\}_{i =1}^k$ consisting of closed flow boxes $\varphi_i : \B_i \rightarrow \Pa_i \times X_i$ such that $\X = \bigcup_{i =1}^k \B_i$ and 
$\inte{\B}_i \cap \inte{\B}_j = \emptyset$ if $i \neq j$. In this context, we can also assume that the plaque 
$\Pa_i$ is a  $\calp$-tile and the change of coordinates is given by 
$$
\varphi_i\scirc\varphi_j^{-1}(x,y) =
\bigl(\varphi_{ij}(x),\gamma_{ij}(y)\bigr),
$$
where the map $\varphi_{ij}$ is a translation from an edge of $\Pa_j$ to an edge of $\Pa_i$. 
%Nevertheless, the laminated space $\X$ admits more general box decompositions as defined in \cite{ALM}. 
In general, any box decomposition is said to be
\emph{well-adapted} (to the $\calp$-tiled structure)  if each plaque $\Pa_i$  is a $\calp$-patch and the
associated total transversal  $\bigsqcup_{i=1}^k X_i$ is a clopen subset
of $X$. For any flow box $\B_i$ in $\calb$, the set $\partial_v \B_i =\varphi_i^{-1}(\partial \Pa_i \times X_i)$ is called the \emph{vertical boundary} of  $\B_i$.

\begin{theorem}[\cite{BBG}]
	\label{thinflac}
	Let  $\X$ be the continuous hull of an aperiodic and repetitive Euclidean
	tiling satisfying the finite pattern condition. Then, for any well-adapted flow box decomposition
	$\calb$ of $\X$, there exists another well-adapted flow box decomposition
	$\mathcal{B'}$ \emph{inflated from $\calb$} having the following properties:
	\medskip 

	\noindent
	i) for each tiling $\calt$  in a box $\B \in \calb$ and in a box
	$\B' \in \mathcal{B}'$, the transversal of $\B'$ through $\calt$ is contained in
	the corresponding transversal of $\B$;
	\medskip

	\noindent
	ii) the vertical boundary of the boxes of $\mathcal{B}'$ is contained in the
	vertical boundary of boxes of $\calb$;
	\medskip 

	 \noindent
	iii) for each box $\B' \in \mathcal{B}'$, there exists a box $\B \in \calb$ such
	that $\B \cap \B' \neq \emptyset$ and $\B \cap \partial_v \B' = \emptyset$. \qed
\end{theorem}

By applying this theorem inductively, we have a sequence of  well-adapted box
decompositions $\calb^{(n)}$ such that  
\begin{list}{}{\leftmargin=14pt}

	\item[1)] $\calb^{(0)}=\calb$, 

	\item[2)]  $\calb^{(n+1)}$ is inflated from  $\calb^{(n)}$ and 

	\item[3)]  $\calb^{(n+1)}$ defines a finite set  $\calp^{(n+1)}$ of
		$\calp^{(n)}$-patches (which contain at least a $\calp^{(n)}$-tile in
		their interiors) and a tiling in $\T(\calp^{(n+1)})$ of each leaf of
		$\X$. 

\end{list}
Since $\calp^{(n+1)}$-tiles are $\calp^{(n)}$-patches and also plaques of $\calb^{(n+1)}$, we shall use the same letter $\Pa$ to denote them. 
\medskip

Let $X^{(n)}$ be the decreasing sequence of total transversals associated to $\calb^{(n)}$.
Given any increasing sequence of integers $N_n$, we can construct a sequence of inflated box decompositions  $\calb^{(n)}$ that such each
Delone set $\cald^{(n)}_{\! _\calt} = L_\calt \cap X^{(n)}$ is $N_n$-separated, i.e. if $\calt_1 \neq \calt_2$ in $\cald^{(n)}_{\! _\calt}$, the distance between $\calt_1$ and $\calt_2$ is bigger or equal than $N_n$.
Such a sequence defines an increasing sequence of CEERs $\calr_n$ on $X$. Indeed, for
each $n\in\N$, the equivalence class $\calr_n[\calt]$ coincides with the \emph{discrete
plaque} $P_n = \Pa_n \cap X$ determined by the plaque 
$\Pa_n$ of $\calb^{(n)}$ passing through $\calt$. We can also see
each discrete plaque $P_n$ as a plaque of a \emph{discrete flow box}
$B = \B \cap X$ and each discrete flow box $B$ as an element
of a discrete box decomposition defined by $\calb$. 

\begin{proposition} 
	\label{minimal}
	The inductive limit $\calr_\infty= \varinjlim \calr_n$ is a minimal open
	AF equiva\-lence subrelation of  $\calr$.
\end{proposition}

\begin{proof} 
	It is clear that  $\calr_\infty$ is an open AF equivalence subrelation of
	$\calr$. On the other hand, in order to show that $\calr_\infty$ is minimal, we must prove that all 
	$\calr_\infty$-equivalence classes meet any open subset $A$ of $X$. But each 
	$\calr_\infty$-equivalence class contains an increasing sequence of discrete plaques  
	$P_ n = \Pa_n \cap X$ where $\Pa_n$ are the plaques of the box
	decompositions $\calb^{(n)}$. Since $\calr$ is minimal, the intersection $A \cap L_\calt$ remains a
	 Delone set quasi-isometric to $L_\calt$, and therefore 
	 $A \cap \Pa_n = A \cap P_ n \neq \emptyset$ for some
	 $n \in \N$. Then $A \cap \calr_\infty[\calt]$ is also nonempty, and thus $\calr_\infty[\calt]$ is dense. 
\end{proof} 

\begin{remark}
The inflation process developed in \cite{BBG} for tilings and tilable laminations has been extended in 
 \cite{ALM} and \cite{L} for transversely Cantor laminations.
\end{remark}

\subsection{Boundary} \label{sSboundary}

Let us start by defining the (discrete) boundary of an EER $\calr_n$:

\begin{definition}
	i) For any tiling $\calt \in X$, let $\partial \calr_n[\calt]$ be the set
	of tilings $\calt' = \calt - v$ such that $v$ is the base point of any
	$\calp$-tile of the patch $\Pa_n$ meeting $\partial \Pa_n$. 
	\medskip 

	\noindent
	ii)  We define the \emph{boundary of $\calr_n$} as the clopen set 
	\[
		\partial \calr_n = \bigcup_{\calt \in X^{(n)}} \partial\calr_n[\calt]
		                 = \bigcup_{\B_n \in \calb^{(n)}} \partial_v B_n
	\]
	where $\partial_v B_n$ is the vertical boundary of the discrete
	flow box $B_n = \B_n \cap X$.
	\medskip 

	\noindent
	iii) Finally, $\partial \calr_\infty = \bigcap_{n\in\N} \partial \calr_n$
	is a meager closed subset of $X$, which we call the \emph{boundary
	of $\calr_\infty$}. 
\end{definition} 
\noindent
Notice that, even if $\calr$ is AF, the boundary $\partial \calr_\infty$ is always nonempty.
\medskip 

For each tiling $\calt \in \partial \calr_\infty$, the $\calr$-class $\calr [\calt]$ separates into several $\calr_\infty$-equivalence classes. By replacing the elements of theses $\calr_\infty$-equivalence classes with the corresponding $\calp$-tiles, we obtain a decomposition of the leaf $L = L_\calt$ passing through $\calt$ into regions having a common boundary $\Gamma = \Gamma_\calt$. In fact, the union of all these common boundaries is a closed subset
 $$
\partial_c \calr_\infty = \bigcap_{n\in\N} \bigcup_{\B_n \in \calb^{(n)}} \partial_v \B_n
$$
of $\X$, which we call  the \emph{continuous boundary} of $\calr_\infty$.
For each $n \in \N$, the closed subset $\bigcup_{\B_n \in \calb^{(n)}} \partial_v \B_n$ of $\X$ intersects $L$ into an infinite graph 
$\Gamma^{(n)} = \Gamma^{(n)}_\calt$ where each edge separates two different tiles of $\calt$ and  is never terminal. 
According to property $(iii)$ in Theorem~\ref{thinflac}, $\Gamma  = \bigcap_{n\in\N} \Gamma^{(n)}$ is an infinite tree without terminal edges. It is clear that $\Gamma$ is acyclic. Morever, if we assume that $\Gamma$ is not connected, then there would be a sequence $\Pa_n$ of inflated $\calb_n$-plaques such that the distance between two disjoint edges would be bounded, but this is not possible. As the box decomposition 
$\calb^{(0)} = \calb$ is finite,  $\Gamma^{(0)}$ has \emph{bounded geometry}, i.e. each vertex has uniformly bounded degree. The same happens with $\Gamma$ and $\Gamma^{(n)}$ for all $n \in \N$. 
%Moreover, the number of connected components of $L-\Gamma$ is equal to the number of $\calr_\infty$-equivalence classes in $\calr [\calt]$, but this number may be infinite.
\medskip 

By replacing $X$ with a different total transversal $\check{X}$ passing through the vertices of $\calt$, we can assume that the corresponding discrete boundary $\cpartial \calr_\infty$ is equal to 
$\partial_c \calr_\infty \cap \check{X}$. However, each point in the new boundary $\cpartial \calr_\infty$ determines four points (in the interior of four different adjacent tiles) of the original boundary  $\partial \calr_\infty$. We resume the above discussion in the following statement:

\begin{proposition}
	\label{boundary}
	The continuous boundary $\partial_c \calr_\infty$ is a closed subset of $\X$
	which admits a natural partition into trees without terminal edges 
	(induced by the TDS of $\X$) with total tranversal $\cpartial \calr_\infty$.
	\end{proposition}

By construction, if $\calt \in \cpartial \calr_\infty$,  the degree of the origin in $\Gamma^{(n)} = \Gamma^{(n)}_\calt$
depends continuously on $\calt$. Thus, we obtain a continuous map $D_n : \cpartial \calr_\infty \to \N$, which is actually defined on the whole transversal $\check{X}^{(n)}$ consisting of all vertices of the $\calp^{(n)}$-tiling inflated from $\calt$. This extends to a continuous map on $\check{X}$ by defining $D_n(\calt) = 2$ when the origin belongs to some edge of the inflated tiling (so that it is a vertex of $\calt$ different from the vertices of the inflated tiling) and $D_n(\calt) = 0$ otherwise (that is, when the origin belongs to the interior of some inflated tile). We denote by $D : \check{X} \to \N$ the infimum of this family of continuous functions, called {\em the degree function of the inflation process}, which verifies the following property: 

\begin{proposition} \label{semi-continuous}
The degree function $D : \check{X} \to \N$ and its restriction $D : \cpartial \calr_\infty \to \N$ to the boundary $\cpartial \calr_\infty$ are upper semi-continuous. \qed 
\end{proposition} 

In order to apply the Absortion Theorem of \cite{GMPS2} in the next step, we 
also need to prove the following result: 

\begin{proposition}
	\label{thin}
	The boundary $\partial \calr_\infty$ is $\calr_\infty$-thin. 
\end{proposition} 

\noindent
In \cite{AGL}, we showed how to derive this result from the type of growth
of the leaves of $\calf$. To do so, we simply adapted the method used by
C.~Series in \cite{Ser} to prove that any measurable foliation with polynomial growth is hyperfinite.
But as in the original proof of the Affability Theorem in \cite{GMPS1}, Proposition~\ref{thin}
will be recovered here using the isoperimetric properties of the pieces
involved in the Robinson inflation. 

\subsection{Absorption}
\label{absorcion}

We give now a first description of the last step of the proof of the Affability Theorem. We 
start by recalling the  Absorption Theorem~4.6 of \cite{GMPS2}, which is a key
ingredient in this proof:

\begin{theorem}[\cite{GMPS2}]
	\label{thabsorption} 
	Let $\calr$ be a minimal AF equivalence relation on the Cantor set $X$. Assume 
	$Y$ is a $\calr$-\'etale and $\calr$-thin closed subset of $X$, and let $\calk$ be a CEER on $Y$ that is transverse to $\calr |_Y$ (i.e. 
	$\calr |_Y \cap \calk= \Delta_Y$ and there is an isomorphism of topological groupoids
	$\varphi :  \calr |_Y \ast \calk \to \calk \ast \calr |_Y$). Then there is a
	homeomorphism $h : X \to X$ such that
	\medskip 
 
	\noindent
	i) $h$ implements an orbit equivalence between the equivalence relation
	$\calr \vee \calk$ generated by $\calr$ and $\calk$, and the AF equivalence relation
	$\calr$;
	\medskip 

	\noindent
	ii) $h(Y)$ is $\calr$-\'etale and $h(Y)$ is $\calr$-thin;
	\medskip 

	\noindent
	iii) $h |_Y$ implements an isomorphism between  $\calr |_Y \vee \calk$ and
	$\calr |_{h(Y)}$.
	\medskip 

	\noindent
	In particular,  $\calr \vee \calk$ is affable. \qed
\end{theorem}

According to Propositions~\ref{minimal},~\ref{boundary}~and~\ref{thin}, the equivalence
relation $\calr_\infty$ fulfills most of the hypotheses of the Absorption Theorem. On the other
hand, since $\calr_\infty$ and $\calr$ are EER on $X$, their graphs (denoted again by $\calr$ and $\calr_\infty$) split into countably many 
clopen bisections of $\calr$ and $\calr_\infty$ respectively. Recall that a {\em bisection} of $\calr_\infty$ (or $\calr$) is the graph of a partial transformation $\varphi : A \to B$ of $\calr_\infty$ (or $\calr$) between two subsets $A$ and $B$ of $X$. If all $\calr$-equivalence classes split into at most two $\calr_\infty$-equivalence classes, we have a global decomposition of $\partial \calr_\infty$ into two clopen subsets $A$ and $B$ and a global transformation $\varphi :  A \to B$ of $\calr |_{\partial \calr_\infty}$ which generates a CEER  $\calk$ transverse to $\calr_\infty |_{\partial \calr_\infty}$. 
In this case,  we can use a former version of the Absorption Theorem (Theorem~4.18  of \cite{GPS}) to see that $\calr = \calr _\infty \vee \calk$ is OE to $\calr _\infty$. But, in general, there may be a partial transformation $\varphi : A \to B$ of $\calr_\infty$ sending a point $\calt_1 \in \partial  \calr_\infty$ to a point $\calt_2 = \varphi(\calt_1) \in \partial  \calr_\infty$ (so that its graph is a clopen bisection of $\calr_\infty$ passing through $(\calt_1,\calt_2) \in \calr_\infty |_{\partial  \calr_\infty}$) and admiting another point $\hat{\calt}_1 \in A \cap \partial  \calr_\infty$ such that $\hat{\calt}_2 = \varphi(\hat{\calt}_1) \notin \partial  \calr_\infty$ (and hence $(\hat{\calt}_1,\hat{\calt}_2) \notin \calr_\infty |_{\partial  \calr_\infty}$). In the other words, the boundary $\partial \calr_\infty$ may not be $\calr_\infty$-\'etale. In order to divide
the boundary $\partial \calr_\infty$ into smaller closed pieces with the same finite number of
$\calr_\infty$-equivalence classes, we can use some special inflation
process, like that described in \cite{GMPS1} using Voronoi tilings, or the one we describe below inspired by  Robinson tilings.

\section{Robinson inflation} 
\label{buenainflacion}

In this section, we define an inflation process for planar tilings which is modeled by
the natural inflation of Robinson tilings (see \cite{GS} and \cite{Ro}). Firstly, the tiling space is replaced with one whose tiles are marked
squares using \cite{SW}. Then we construct a family of inflated flow boxes whose plaques are squares
of side $N_1$, that is maximal in the sense that there is no space for any other square of side $N_1$. Finally, we replace this partial box decomposition with a complete box decomposition in such a way that the isoperimetric ratios of the plaques
are still good. In this way, we recurrently obtain an inflation process such
that all $\calr$-equivalence classes split at most into $4$ different
$\calr_\infty$-equivalence classes. However, as in the general case, the EER $\calr_\infty$ is still too coarse to make \'etale its boundary $\partial \calr_\infty$. Thus, in Section~\ref{filtrating}, we shall decorate Robinson plaques to obtain a minimal open AF equivalence subrelation 
$\hat{\calr}_\infty$ such that $\partial \calr_\infty$ becomes a $\hat{\calr}_\infty$-thin and $\hat{\calr}_\infty$-\'etale closed subset of $X$. 
Using similar absorption techniques to those used in \cite{GMPS1}, we shall reduce the number of $\hat{\calr}_\infty$-equivalence classes in which $\calr$-equivalence classes are decomposed until $\calr$ is proven to be affable.

\subsection{Sadun-Williams reduction to square tilings} 

Let $\T(\calp)$ be the foliated space of all planar tilings $\calt$
constructed  from a finite set of prototiles $\mathcal P$. Let us recall that tiles
are obtained by translation from a finite number of polygons, which are touching face-to-face.
For any repetitive tiling $\calt \in \T(\calp)$, the continuous hull of
$\calt$ is the minimal closed subset $\X = \overline{L}_\calt$. 
The induced TDS is transversely modeled by the set $X$ of elements of $\X$ where the origin belongs to $\mathcal{D}_\calt$. If $\calt$ is also aperiodic, then $X$ is homeomorphic to the Cantor set.

\begin{theorem}[\cite{SW}]
	The continuous hull of any planar tiling is OE to the
	continuous hull of a tiling whose tiles are marked squares. \qed
\end{theorem}

According to this theorem, we assume $\calp$ is a finite set of marked
squares and we consider an aperiodic and repetitive tiling $\calt$ in
$\X$. Furthermore, all leaves $\X$  are
endowed with the $\max$-distance.

\subsection{Constructing an intermediate tiling} 
Now, let us start by inflating some tiles in the usual sense: there is a finite number of  of flow boxes $\B_{1,i} \cong \Pa_{1,i} \times X_{1,i}$, $i = 1, \dots, k_1$,  such that the plaques $\Pa_{1,i}$ are squares of  side $N_1$ and the transversal
$X_1 =  \bigcup_{i=1}^{k_1} X_{1,i}$ is $N_1$-dense in $\X$ (i.e. any ball of radius $N_1$ meets $X_1$) with respect to the longitudinal $\max$-distance. 
We say that $\calb_1 = \{ \B_{1,i} \}_{i=1}^{k_1}$ is a \emph{partial box decomposition} of $\X$, and 
we  write $\calp_1 = \{ \Pa_{1,i} \}_{i=1}^{k_1}$.  The first step in order to inflate tilings in $\X$ is to replace $\calb_1$ with a true box decomposition $\pb_1$. To do so, we  need some preliminaries: 

\begin{figure}
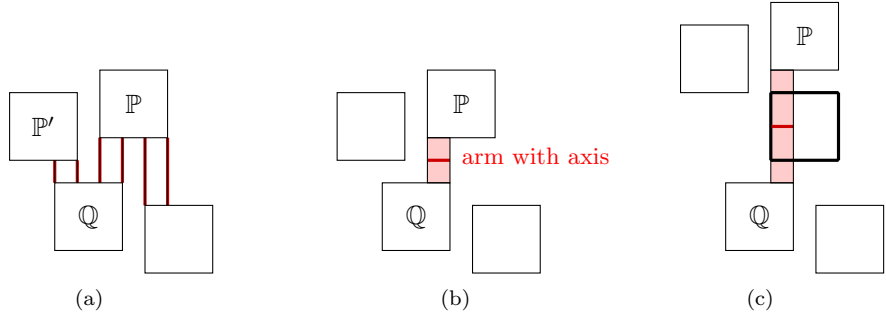
%[H]
\subfigure[]{
\begin{mosaicoPBC}[scale=.3]
\Va{4}\Pl{3}\Nl
\Pl{3} \Nl{2}
\Va{4}\Br|{0}{2}\Va{1}\Br|{0}{2}\Va{1}\Br|{0}{3}\Va{1}\Br|{0}{3}\Nl{1}
\Va{2}\Br|{0}{1}\Va{1}\Br|{0}{1}\Nl
\Va{2}\Pl{3}\Nl
\Va{6}\Pl{3}Nl
\node at (5.5,1.5){$\Pa$};
\node at (1.5,2.5){$\Pa'$};
\node at (3.5,6.5){$\Q$};
\end{mosaicoPBC}
\label{arm1}
}
\hspace{2em}
\subfigure[]{
\begin{mosaicoPBC}[scale=.3]
\Va{4}\Pl{3}\Nl
\Pl{3} \Nl{2}
\Va{4}\Br-{1}{2}\Nl{2}
\Va{2}\Pl{3}\Nl
\Va{6}\Pl{3}Nl
\node at (5.5,1.5){$\Pa$};
\node at (3.5,6.5){$\Q$};
\node at (8.8,3.8) [red] {$\mbox{\small arm with axis}$};
\end{mosaicoPBC}
\label{arm2}
}
\subfigure[]{
\begin{mosaicoPBC}[scale=.3]
\Va{4}\Pl{3}\Nl
\Pl{3}\Nl{2}
\Va{4}\Br-{1}{5}\Nl{1}
\Va{4}{\red \Pl{3}}\Nl{4}
\Va{2}\Pl{3}\Nl
\Va{6}\Pl{3}\Nl
\draw[very thick] (4,4) -- (7,4);
\draw[very thick] (7,4) -- (7,7);
\draw[very thick] (7,7) -- (4,7);
\draw[very thick] (4,7) -- (4,4);
\node at (5.5,1.5){$\Pa$};
\node at (3.5,9.5){$\Q$};
\end{mosaicoPBC}
\label{arm3}
}
\caption{Arms and axes}
\label{arm}
\end{figure}

\begin{definition}
	Let $\Pa$ and $\Q$ be two plaques of $\calb_1$ contained in the same
	leaf of  $\X$. We say that $\Q$ is a \emph{neighbor of $\Pa$} if the
	orthogonal distance from an edge of $\Pa$ to $\Q$ is the minimum of the distance
	from this edge of $\Pa$ to another plaque, see Figure~\ref{arm1}.  In this case, the union of the 
	orthogonal segments which realize the distance between $\Pa$ and $\Q$ is called an 
	\emph{arm} (of the complement of the partial tiling defined by $\calb_1$), see 
	Figure~\ref{arm2}. Such an arm contains a segment, called the \emph{axis}, which is 
	parallel and equidistant to the corresponding edges of $\Pa$ and $\Q$. Also note that 
	$\Pa$ is not always a neighbor of  $\Q$, even though $\Q$ is a neighbor of $\Pa$, 
	see Figure~\ref{arm1}.
\end{definition} 

\begin{lemma} 
	Any arm has a length and width less than or equal to $N_1$.
\end{lemma}

\begin{proof}
	Obviously the length of any arm is bounded by the side length $N_1$ of
	the plaques of $\calb_1$. On the other hand, if the width of an arm (between
	two plaques $\Pa$ and $\Q$) were larger than $N_1$, there would be space for
	another plaque (between $\Pa$ and $\Q$) and $X_1$ would not be $N_1$-dense, 
	see Figure~\ref{arm3}.
\end{proof} 

\begin{lemma}
	For each tiling $\calt \in \X$, let us consider the union of the plaques of the partial box 
	decomposition $\calb_1$ and the corresponding arms. Each connected component 
	of the complement of this union is a marked rectangle, called a \emph{cross}, 
	with sides of a length less than or equal to $2N_1$. 
\end{lemma}

\begin{proof}
	Since each plaque of $\calb_1$ has at most $4$ neighbors, the boundary of each
	connected component meets $3$ or $4$ plaques of $\calb_1$ and $3$ or $4$ arms
	between these plaques. In both cases, the connected component is a rectangle
	of side length at most $2N_1$, as can be see in Figures~\ref{nondeg4}
	and \ref{nondeg3}.
\end{proof}

\begin{remark}
	It is also interesting to note that there could be degenerate arms and crosses,
	such as those illustrated in the Figure~\ref{deg}.
\end{remark}

	\begin{figure}
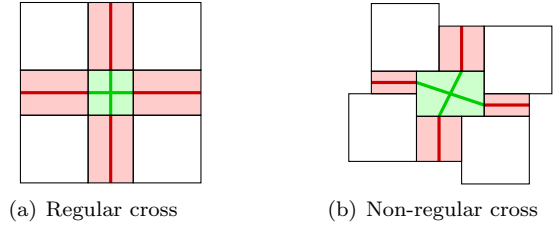
%[H]
		\subfigure[Regular cross]{
			\begin{mosaicoPBC}[scale=.3]
				\Pl{3}\Br|{2}{3}\Pl{3}\Nl{3}
				\Br-{3}{2}\Cr{2}{2}\Br-{3}{2}\Nl{2}
				\Pl{3}\Br|{2}{3}\Pl{3}
			\end{mosaicoPBC}
			\label{regarm4}
		}
		\hspace{1cm}
		\subfigure[Non-regular cross]{
			\begin{mosaicoPBC}[scale=.3]
				\Va{1}\Pl{3}\Nl{1}
				\Va{4}\Br|{2}{2}\Pl{3}\Nl{2}
				\Va{1}\Br-{2}{1}\Cr{3}{2}\Nl
				\Pl{3}\Va{3}\Br-{2}{1}\Nl
				\Va{3}\Br|{2}{2}\Pl{3}
			\end{mosaicoPBC}
			\label{nonregarm4}
		} 
		\caption{Crosses with $4$ arms}
		\label{nondeg4}
	\end{figure}

	\begin{figure}
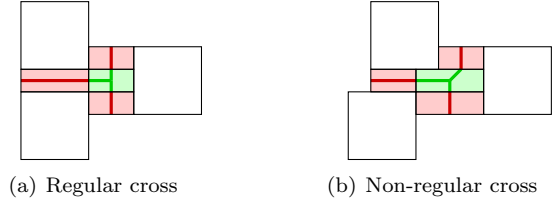
%[H]
		\subfigure[Regular cross]{
			\begin{mosaicoPBC}[scale=.3]
				\Pl{3}\Nl{2}
				\Va{3}\Br|{2}{1}\Pl{3}\Nl
				\Br-{3}{1}\Cr{2}{1}\Nl
				\Pl{3}\Br|{2}{1}\Nl
			\end{mosaicoPBC}
			\label{regarm3}
		}
		\hspace{1cm}
		\subfigure[Non-regular cross]{
			\begin{mosaicoPBC}[scale=.3]
				\Va{1}\Pl{3}\Nl{2}
				\Va{4}\Br|{2}{1}\Pl{3}\Nl
				\Va{1}\Br-{2}{1}\Cr{3}{1}\Nl
				\Pl{3}\Br|{3}{1}
			\end{mosaicoPBC}
			\label{nonregarm3}
		}
		\caption{Crosses with $3$ arms}
		\label{nondeg3}
	\end{figure}

\begin{figure}%[H]
\subfigure[$4$ arms]{
\begin{mosaicoPBC}[scale=.3]
\Pl{3}\Br|{0}{3}\Pl{3}\Nl{3}
\Br-{3}{2}\Cr{0}{2}\Br-{3}{2}\Nl{2}
\Pl{3}\Br|{0}{3}\Pl{3}
\draw[green,very thick] (2.9,3) -- (2.9,5);
\draw[green,very thick] (3.1,3) -- (3.1,5);
\draw[red,very thick] (2.9,0) -- (2.9,3);
\draw[red,very thick] (3.1,0) -- (3.1,3);
\draw[red,very thick] (2.9,5) -- (2.9,8);
\draw[red,very thick] (3.1,5) -- (3.1,8);
\end{mosaicoPBC}
\label{deg3exit}
}
\hspace{2cm}
\subfigure[$3$ arms]{
\begin{mosaicoPBC}[scale=.075]
\Pl{12}\Br|{0}{12}\Pl{12}\Nl{12}
\Va{6}\Br-{6}{8}\Cr{0}{8}\Br-{6}{8}\Nl{8}
\Va{6}\Pl{12}
\draw[red,very thick] (11.6,0) -- (11.6,12);
\draw[red,very thick] (12.4,0) -- (12.4,12);
\draw[green,very thick] (11.6,12) -- (11.6,20);
\draw[green,very thick] (12.4,12) -- (12.4,20);
\end{mosaicoPBC}
\label{deg2exit}
}
%\caption{Degenerate cross and arms}
%\label{deg}
\end{figure}

\begin{figure}
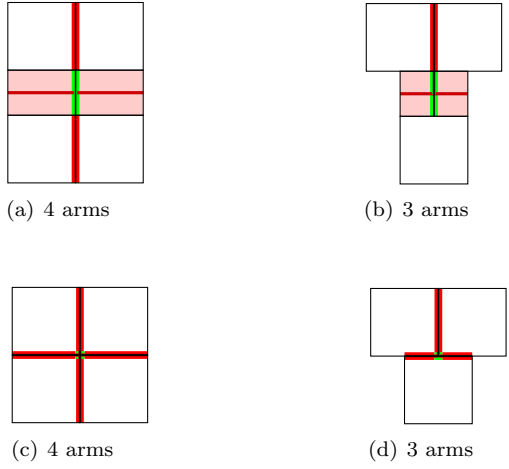
%[H]
\subfigure[$4$ arms]{
\begin{mosaicoPBC}[scale=.3]
\Pl{3}\Br|{0}{3}\Pl{3}\Nl{3}
\Br-{3}{0}\Cr{0}{0}\Br-{3}{0}\Nl{0}
\Pl{3}\Br|{0}{3}\Pl{3}
%\draw[green,very thick] (2.9,2.9) -- (2.9,3.1);
%\draw[green,very thick] (3.1,2.9) -- (3.1,3.1);
\draw[red,very thick] (2.9,0) -- (2.9,2.85);
\draw[red,very thick] (3.1,0) -- (3.1,2.85);
\draw[red,very thick] (2.9,3.1) -- (2.9,6);
\draw[red,very thick] (3.1,3.1) -- (3.1,6);
\draw[red,very thick] (0,2.9) -- (2.8,2.9);
\draw[red,very thick] (0,3.1) -- (2.8,3.1);
\draw[red,very thick] (3.2,2.9) -- (6,2.9);
\draw[red,very thick] (3.2,3.1) -- (6,3.1);
\draw[green,very thick] (2.8,2.9) -- (3.2,2.9);
\draw[green,very thick] (2.8,3.1) -- (3.2,3.1);
\end{mosaicoPBC}
\label{deg1exit4}
}
\hspace{2cm}
\subfigure[$3$ arms]{
\begin{mosaicoPBC}[scale=.075]
\Pl{12}\Br|{0}{12}\Pl{12}\Nl{12}
\Va{6}\Br-{6}{0}\Cr{0}{0}\Br-{6}{0}\Nl{0}
\Va{6}\Pl{12}
%\draw[green,very thick] (11.6,11.4) -- (11.6,12.6);
%\draw[green,very thick] (12.4,11.4) -- (12.4,12.6);
\draw[red,very thick] (11.6,0) -- (11.6,11.4);
\draw[red,very thick] (12.4,0) -- (12.4,11.4);
\draw[red,very thick] (6,11.6) -- (11.3,11.6);
\draw[red,very thick] (6,12.4) -- (11.3,12.4);
\draw[red,very thick] (12.7,12.4) -- (18,12.4);
\draw[red,very thick] (12.7,11.6) -- (18,11.6);
\draw[green,very thick] (11.3,11.6) -- (12.7,11.6);
\draw[green,very thick] (11.3,12.4) -- (12.7,12.4);
\end{mosaicoPBC}
\label{deg1exit3}
}
\caption{Degenerate crosses and arms}
\label{deg}
\end{figure}

\subsection{Robinson inflation.}
Using arms and crosses, we can replace the partial box decomposition $\calb_1$
with a complete box decomposition $\pb_1$. In this  step, we modify
this decomposition so that it becomes inflated from the initial decomposition $\calb$
by unit square plaques. We start by introducing some definitions: 

\begin{definition}
	i) The middle point of the intersection of each arm with a cross is called
	an \emph{exit point} of the cross. Each side of a cross contains at most
	$1$ exit point and each nondegenerate cross has at least $3$
	 exit points, see  Figures~\ref{nondeg4}~and~\ref{nondeg3}. They are \emph{positive} or 
	 \emph{negative} end points of	the axes (relatively to the usual positive orientation). 
	 If the cross is degenerated, there are still at least $3$ exit points counting multiplicities, see 
	Figure~\ref{deg}. 
	\medskip 

	\noindent
	ii) Each cross is decorated with a graph which is obtained by joining the center of mass with the
	exit points. In the nondegenerate case, this decoration separates the cross into $3$
	or $4$ regions. In the degenerate case, the cross reduces to a common
	side of two different arms (decorated with the middle point contained in the two
	axes) or a single point (which is the intersection of the degenerate arms). Each
	of these regions is called a \emph{cross-sector}. We also use the term
	\emph{cross-sector} to refer to the union of $\calp$-tiles that meet the original cross-sector. 
	As explained in \cite{BBG}, it is irrelevant where $\calp$-tiles which meet several sectors are  
	included, although for simplicity we shall assume such a  $\calp$-tiles are included 
	in the sectors pointing right and upward. 
	\medskip 

	\noindent
	iii) Let $\Pa$ be a plaque of the partial box decomposition $\calb_1$. We denote by
	$\pp$ the union of $\Pa$ with the half-arms and the cross-sectors  meeting $\Pa$. Replacing 
	these half-arms and cross-sectors with corresponding half-arms and cross-sectors made of 
	$\calp$-tiles as above, we obtain another planar set $\ppp$. Now  $\ppp$  is a tile of a planar tiling 
	$\calt''$ inflated from $\calt$  and a plaque of a box decomposition $\ppb_1$ inflated 
	from $\calb$. We  call this process  \emph{Robinson inflation}.
\end{definition}

\begin{lemma}
	\label{isoperimetric1}
	Any tile $\pp$ and any inflated tile $\ppp$ contain a square of side $N_1$, and they are contained 
	in a square of side $3N_1$. Thus their areas  $\area(\pp)$ and  $\area(\ppp)$ are comprised 
	between $N_1^2$ and $9 N_1^2$.
	If $\pp$ and $\ppp$ are associated to the same plaque $\Pa$ of the partial box
	decomposition $\calb_1$, then 
	\[
		\area(\ppp) \geq \area(\pp) - \longi(\partial \pp) \area(\Pa_0) =
			\area(\pp) - \longi(\partial \pp) 
	\]
	and
	\[
		\longi(\partial \ppp) \leq \longi(\partial \pp) \longi(\partial \Pa_0) \leq 64 N_1
	\]
	where $\Pa_0$ is the square prototile of $\calt$. 
\end{lemma}

\begin{proof}
	Firstly, since $\ppp$ is obtained from $\pp$ by adding or removing $\calp$-tiles, we
	have that
	\[
		\area(\ppp) \geq \area(\pp) - \longi(\partial \pp) \area(\partial \Pa_0)
	\]
	and 
	\[
		\longi(\partial \ppp) \leq \longi(\partial \pp) \longi(\partial \Pa_0) \medskip 
	\]
	where $\area(\Pa_0) = 1$ and $\longi(\partial \Pa_0) = 4$. On the other hand,
	when we replace $\Pa$ with $\pp$, the side length increases by at most
	$2\sqrt{2}N_1 \leq 3 N_1$,  where $2\sqrt{2}N_1$ is the product of the maximum 
	number of nondegenerate crosses that meet each edge of $\Pa$ (equal to $2$) and the 
	half-diagonal of the square of side $2N_1$ (equal to $\sqrt{2}N_1$). Thus
	$\longi(\partial \pp) \leq 4(N_1+3N_1) = 16 N_1$
	and we have finished.
\end{proof}

Arguing inductively, we have the following theorem: 

\begin{theorem}
	\label{Robinsoninflation}
	There is a sequence of  box decompositions $\ppb_n$ such that  $\ppb_0=\calb$ and
	$\ppb_{n+1}$ is obtained by Robinson inflation of $\ppb_{n}$. \qed
\end{theorem} 

We denote by $\calp'_n$ and $\calp''_n$ the sets of the plaques of $\pb_n$ and $\ppb_n$ respectively. Then $\pb_n$ and $\ppb_n$ induce tilings $\calt'_n \in \T(\calp'_n)$ and 
 $\calt''_n \in \T(\calp''_n)$ on each leaf of $\X$. By construction, each plaque $\pp_{n+1}$ of the box decomposition $\pb_{n+1}$  is a $\calp'_n$-patch, which is constructed from
a square of side $N_{n+1}$ and contained in a square of side $3N_{n+1}$. 
In the same way, each plaque $\ppp_{n+1}$ of the inflated box decomposition $\ppb_{n+1}$ is a
 $\calp''_n$-patch. Thus, the tiling $\calt''_{n+1} \in \T(\calp''_{n+1})$ induced by $\ppb_{n+1}$ is 
obtained by Robinson inflation from the tiling
$\calt''_n \in \T(\calp''_n)$ induced by $\ppb_n$. As in the proof of
Lemma~\ref{isoperimetric1}, we can see that
\[
	\area(\ppp_{n+1}) \geq \area(\pp_{n+1}) - \longi(\partial \pp_{n+1}) A_n
\]
and
\[
	\longi(\partial \ppp_{n+1}) \leq \longi(\partial \pp_{n+1}) L_n \medskip
\]
where $A_n = \max \{\, \area(\ppp_n) \tq \ppp_n \in \calp''_n \,\}$ and 
$L_n = \max \{\, \longi(\ppp_n) \tq \ppp_n \in \calp''_n \,\}$. On the other hand, 
we have that
\[
	\area(\pp_{n+1}) \geq N_{n+1}^2 \quad \mbox{and} \quad 
		\longi(\partial \pp_{n+1}) \leq 16 N_{n+1}.
\]
Finally, we can assume that $L_n \leq A_n$ when $n$ is chosen large enough.

\begin{proposition}
	\label{isoperimetric2}
	There is a sequence of positive integers $N_n$ such that the isoperimetric ratio 
	\[
		\frac{\longi(\partial \ppp_n)}{\area(\ppp_n)} \to 0
	\]
	as $n\to \infty$. 
\end{proposition}

\begin{proof}
	According to the previous inequalities, if we choose $N_{n+1} \geq N_n^3$,
	then the isoperimetric ratio
	
	\begin{align*}
		\frac{\longi(\partial \ppp_{n+1})}{\area(\ppp_{n+1})} & \leq  
		\frac{\longi(\partial \pp_{n+1})L_n}{\area(\pp_{n+1}) - \longi(\partial \pp_{n+1}) A_n } \\
		& \leq  \frac{\longi(\partial \pp_{n+1})A_n}{\area(\pp_{n+1}) - \longi(\partial \pp_{n+1}) A_n } \\
		& =  \frac{1}{\frac{\area(\pp_{n+1})}{ \longi(\partial \pp)}  \frac{1}{A_n} - 1}  \\ 
		& \leq   \frac{1}{ \frac{N_{n+1}^2}{16N_{n+1}}\frac{1}{9N_n^2} - 1} 
		 =  \frac{(12N_n)^2}{N_{n+1} - (12N_n)^2}
	\end{align*}
	converges to $0$.
\end{proof} 

\section{Properties of the boundary for Robinson inflation} \label{sboundary}

In this section, we replace the initial  box decompositions $\calb^{(n)}$ by
the box decompositions $\ppb_n$ provided by Theorem~\ref{Robinsoninflation}, but
we  keep all of the notations introduced in Section~\ref{schema}. Thus, according to
Proposition~\ref{minimal}, we have a minimal open AF equivalence subrelation 
$\calr_\infty= \varinjlim \calr_n$ of the equivalence relation $\calr$. Our first aim is to prove Proposition~\ref{thin}: 

\begin{proof}[Proof of Proposition~\ref{thin}]
	Let $\mu$ be a $\calr_\infty$-invariant probability measure on $X$. For each discrete flow box 
	$B''_n \cong P''_n \times X''_n$ defined from an element of $\ppb_n$, we have that
	\[
		\mu(B''_n) =  \card P''_n \mu(X''_n)
	\quad \mbox{ and } \quad 
		\mu(\partial_v B_n)= \card \partial P''_n \mu(X''_n)
	\]
	where $\card P''_n$ and $\card \partial P''_n$
	denote the number of elements of the discrete plaque $P''_n$ and its
	boundary $\partial P''_n$. For each $n \in \N$, it follows
	that 
	\begin{align*}
		\mu(\partial \calr_n) & = \sum_{\B''_n \in \ppb_n} \mu(\partial_v B''_n) \\
		& = \sum_{\B''_n \in \ppb_n} \card \partial P''_n \mu(X''_n) \\
		& = \sum_{\B''_n \in \ppb_n} \frac{\card \partial P''_n}{\card P''_n} \mu(B''_n) \\
		& \leq  \max_{\Pa''_n \in \calp''_n}  \Big\{ \frac{\card \partial P''_n}{\card P''_n} \Big\} 
			\sum_{\B''_n \in \ppb_n} \mu(B''_n)  \\
		& =  \max_{\Pa''_n \in \calp''_n} \Big\{\frac{\card \partial P''_n}{\card P''_n} \Big\} \, \mu(X) 
			=   \max_{\Pa''_n \in \calp''_n} \Big\{\frac{\card \partial P''_n}{\card P''_n} \Big\} 
	\end{align*}
	Since the leaves of the continuous hull $\X$ are quasi-isometric to the 
	$\calr$-equivalence classes in $X$, 
	Proposition~\ref{isoperimetric2} implies that 
	\[
		\lim_{n \to \infty} \frac{\card \partial P''_n}{\card P''_n} = 0
	\]
	for all $\Pa''_n \in \calp''_n$. In fact, by replacing the $\max$-metric along the leaves with
	the discrete metric (defined as the minimum length of the paths of $\calp$-tiles
	connecting two points), we can assume that
	$\longi(\partial \ppp_n) = \card \partial P''_n$ and $\area(\ppp_n) = \card P''_n$.
	Anyway, we have that
	\[
		\mu(\partial \calr_\infty)= \lim_{n\rightarrow \infty} \mu(\partial \calr_n)  = 0
	\]
	and then $\partial \calr_\infty$ is $\calr_\infty$-thin. 
\end{proof}

As announced, we are interested in a key property of the boundary:

\begin{proposition}
	\label{finiteness}
	Any $\calr$-equivalence class separates into at most four $\calr_\infty$-equiva\-lence classes. 
\end{proposition}

We shall demonstrate that each tree $\Gamma = \Gamma_\calt$ contained in the continuous boundary $\partial_c \calr_\infty$  separates the corresponding leaf $L = L_\calt \subset \X$ into at most $4$ connected components. Firstly, 
%in order to prove Proposition~\ref{finiteness}, 
we need to introduce some definitions and distinguish some  cases. Let us recall that we have a sequence of box decompositions $\pb_n$ of $\X$ whose plaques are obtained from  squares $\Pa_n$ of side $N_n$, arms $A_n$ and crosses $C_n$. In the rest of section, we shall assume all these plaques belong to the same leaf $L$ of $\X$.

\begin{definition}
We call \emph{virtual arm} of the inflation process a sequence of arms $A_n$ whose axes $a_n$ are contained in a horizontal or vertical ribbon of constant width. Similarly, a \emph{virtual cross} is a sequence of crosses $C_n$ whose centers of mass $c_n$ are contained in a square of constant side.  Note that the axes and centers of mass can oscillate in the interior of the ribbon or square, see Figure~\ref{virtual}. 
\end{definition}

\begin{proof}[Proof of Proposition~\ref{finiteness}] To prove this result, we distinguish some cases:
\medskip 

\paragraph*{\em {\bf Case 1)} There are neither axes nor virtual crosses.} 
In this case, $L$ does not intersect $\partial_c \calr_\infty$. This
means that $\calr_\infty[\calt] = \calr[\calt]$ for all $\calt \in L \cap X$. 
\medskip 

\paragraph*{\em {\bf Case 2)} There is a virtual axis, but there are no virtual crosses.} By definition, there is a sequence of arms $A_n$ whose axes $a_n$ remain in the interior of a ribbon of constant width. There are now two possibilities: 
\medskip 

\paragraph*{\em {\bf Subcase 2.1)} The positive and negative end points of the axes $a_n$ do not remain in the interior of  a square of constant side.} Thus, all the axes $a_n$ grow in both opposite (horizontal or vertical) directions determined by the axis of the ribbon. For simplicity, we  say that the axes $a_n$ grow \emph{on the left and right} in the horizontal case and \emph{upward and downward} in the vertical case. In both cases, since $\partial_c \calr_\infty$ intersect $L$ in a tree without terminal edges (see Proposition~\ref{boundary}), which is contained in the horizontal or vertical ribbon,  this intersection separates $L$ into $2$ connected components, and therefore $\calr[\calt]$ separates into two $\calr_\infty$-equivalence classes. 
\medskip 

\begin{figure}
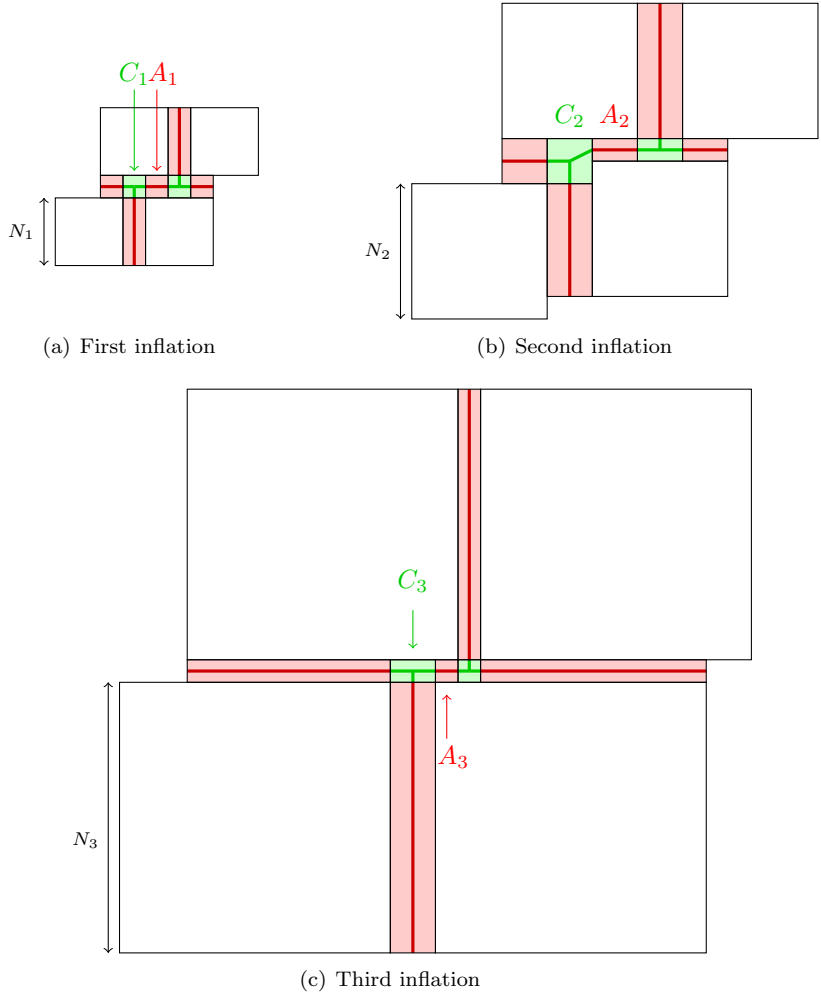

\begin{minipage}{1.4in}
\centering
\begin{mosaicoPBC}[scale=.3]
\Va{2}\Pl{3}\Br|{1}{3}\Pl{3}\Nl{3}
\Va{2}\Br-{1}{1}\Cr{1}{1}\Br-{1}{1}\Cr{1}{1}\Br-{1}{1}\Nl{1}
\Pl{3}\Br|{1}{3}\Pl{3}\Nl
\node at (3.5,-1.5) [green!80!black] {$C_1$};
\path[->]  (3.5,-0.8) edge[green!80!black] (3.5,2.8);
\node at (4.8,-1.5) [red] {$A_1$};
\path[->]  (4.5,-0.8) edge[red] (4.5,2.8);
\path[<->] (-0.5,4) edge (-0.5,7);
\node at (-1.5,5.5) {$\scriptstyle N_1$};
%\path[<->] (1.5,3) edge (1.5,4);
%\node at (1,3.5) {$\scriptstyle 2$};
%ALVARO'S VERSION
%	\node at (3.5,-1.5) [ColorCruzB] {$C_1$};
%	\path[->]  (3.5,-0.8) edge[ColorCruzB] (3.5,2.8);
%	\node at (4.8,-1.5) [ColorBrazoB] {$A_1$};
%	\path[->]  (4.5,-0.8) edge[ColorBrazoB] (4.5,2.8);
%	\path[<->] (-0.5,4) edge (-0.5,7);
%	\node at (-1.5,5.5) {$\scriptstyle N_1$};
\end{mosaicoPBC}
\end{minipage}
\hspace{1cm}
\begin{minipage}{2.2in}
\centering
\begin{mosaicoPBC}[scale=.3]
\Va{4}\Pl{6}\Br|{2}{6}\Pl{6}\Nl{6}
\Va{4}\Br-{2}{2}\Cr{2}{2}\Br-{2}{1}\Cr{2}{1}\Br-{2}{1}\Nl{1}
\Va{8}\Pl{6}\Nl
\Pl{6}\Br|{2}{5}\Nl
\node at (9,5) [red] {$A_2$};
\node at (7,5) [green!80!black] {$C_2$};
\path[<->] (-0.5,8) edge (-0.5,14);
\node at (-1.5,11) {$\scriptstyle N_2$};
%\path[<->] (3.5,6) edge (3.5,7);
%\node at (2.3,6.5) {$\scriptstyle 2N_0$};
%ALVARO'S VERSION
%	\node at (9,5) [ColorCruzB] {$A_2$};
%	\node at (7,5) [ColorBrazoB] {$C_2$};
%	\path[<->] (-0.5,8) edge (-0.5,14);
%	\node at (-1.5,11) {$\scriptstyle N_2$};
\end{mosaicoPBC}
\end{minipage}
\vspace*{-4ex}
\begin{figure}[H]
\subfigure[First inflation]{\hspace*{1.4in}
\label{virtual1}
}
\hspace{1cm}
\subfigure[Second inflation]{\hspace*{2.2in} 
\label{virtual2}
 } 
 \subfigure[Third inflation]{
 \begin{mosaicoPBC}[scale=.3]
\Va{3}\Pl{12}\Br|{1}{12}\Pl{12}\Nl{12}
\Va{3}\Br-{9}{1}\Cr{2}{1}\Br-{1}{1}\Cr{1}{1}\Br-{10}{1}\Nl{1}
\Pl{12}\Br|{2}{12}\Pl{12}\Nl
\node at (14.8,16.3) [red] {$A_3$};
\path[->]  (14.5,15.5) edge[red] (14.5,13.55);
\node at (13,8.5) [green!80!black] {$C_3$};
\path[->]  (13,9.8) edge[green!80!black] (13,11.5);
\path[<->] (-0.5,13) edge (-0.5,25);
\node at (-1.5,20) {$\scriptstyle N_3$};
%\path[<->] (26.5,12) edge (26.5,13);
%\node at (27.8,12.5) {$\scriptstyle 2N_2$};
%ALVARO'S VERSION
%	\node at (14.8,16.3) [ColorCruzB] {$A_3$};
%	\path[->]  (14.5,15.5) edge[ColorCruzB] (14.5,13.55);
%	\node at (13,8.5) [ColorBrazoB] {$C_3$};
%	\path[->]  (13,9.8) edge[ColorBrazoB] (13,11.5);
%	\path[<->] (-0.5,13) edge (-0.5,25);
%	\node at (-1.5,20) {$\scriptstyle N_3$};
\end{mosaicoPBC}
\label{virtual3}
 } 
\caption{Virtual arms and crosses}
\label{virtual}
\vspace{-2ex}
\end{figure}
\end{figure}

\paragraph*{\em {\bf Subcase 2.2)} Positive or negative end points remain in the interior of a square of constant side.} 
Let us assume first that only negative end points remain in the interior of a square of constant side. In this case, since there are no virtual crosses, there are bigger and bigger crosses $C_n$ whose exit points on the right (resp. up) side coincide with the negative end points of the axes $a_n$ of the horizontal (resp. vertical) arms $A_n$. Thus, the decorations of the crosses $C_n$ must grow on the left 
(i.e. the negative direction of the axis of the horizontal ribbon) or downward (i.e. the negative direction of the axis of the vertical ribbon). Now, by joining the axis $a_n$ with the horizontal (resp. vertical) edge of the decoration of $C_n$, this subcase is reduced to the above one. A similar argument may be used when only positive end points remain in the  interior of a square of constant side. Finally, if both negative and positive end points remain in the interior of such a square, then there is a virtual arm (defined by a sequence of arms of bounded length) connecting bigger and bigger crosses $C_n$ and $C'_n$ on the left and right (resp. upward and downward) side. They must have $3$ exit points, and grow in the opposite directions. Reasoning on the left and right (resp. upward and downward) as in the previous cases, we come to the same conclusion. 
\medskip

\paragraph*{\em {\bf Case 3)} There is a single virtual cross.} Thus, there is a sequence of crosses $C_n$ whose centers of mass $c_n$ belong to a square of constant side. As before, we distinguish two subcases: 

\begin{figure}
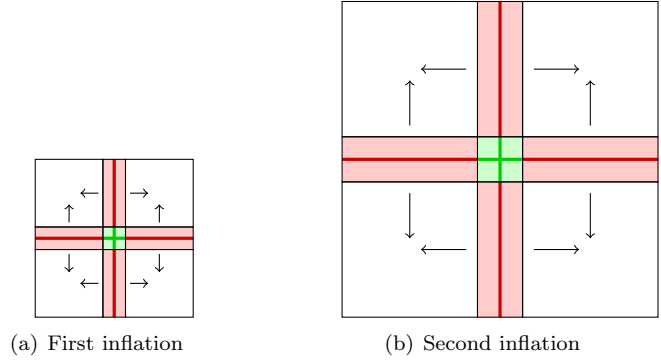
%[H]
\subfigure[First inflation]{
\begin{mosaicoPBC}[scale=.3]
\Pl{3}\Br|{1}{3}\Pl{3}\Nl{3}
\Br-{3}{1}\Cr{1}{1}\Br-{3}{1}\Nl{1}
\Pl{3}\Br|{1}{3}\Pl{3}\Nl
\path[->] (2.8,1.5) edge (2,1.5);
\path[->] (2.8,5.5) edge (2,5.5);
\path[->] (4.2,1.5) edge (5,1.5);
\path[->] (4.2,5.5) edge (5,5.5);
\path[->] (1.5,2.8) edge (1.5,2);
\path[->] (5.5,2.8) edge (5.5,2);
\path[->] (1.5,4.2) edge (1.5,5);
\path[->] (5.5,4.2) edge (5.5,5);
\end{mosaicoPBC}
\label{vc4a1}
}
\hspace{1cm}
\subfigure[Second inflation]{
\begin{mosaicoPBC}[scale=.3]
\Pl{6}\Br|{2}{6}\Pl{6}\Nl{6}
\Br-{6}{2}\Cr{2}{2}\Br-{6}{2}\Nl{2}
\Pl{6}\Br|{2}{6}\Pl{6}\Nl
\path[->] (5.5,3) edge (3.5,3);
\path[->] (5.5,11) edge (3.5,11);
\path[->] (8.5,3) edge (10.5,3);
\path[->] (8.5,11) edge (10.5,11);
\path[->] (3,5.5) edge (3,3.5);
\path[->] (11,5.5) edge (11,3.5);
\path[->] (3,8.5) edge (3,10.5);
\path[->] (11,8.5) edge (11,10.5);
\end{mosaicoPBC}
\label{vc4a2}
 } 
\caption{Virtual cross with $4$ exit points}
\label{vc4a}
\end{figure}

\begin{figure}
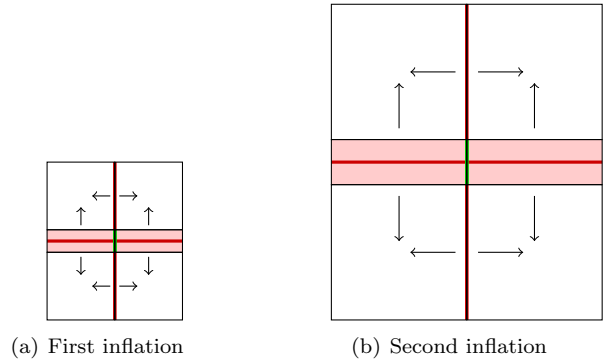
%[H]
\subfigure[First inflation]{
\begin{mosaicoPBC}[scale=.3]
\Pl{3}\Br|{0}{3}\Pl{3}\Nl{3}
\Br-{3}{1}\Cr{0}{1}\Br-{3}{1}\Nl{1}
\Pl{3}\Br|{0}{3}\Pl{3}\Nl
\path[->] (2.8,1.5) edge (2,1.5);
\path[->] (2.8,5.5) edge (2,5.5);
\path[->] (3.2,1.5) edge (4,1.5);
\path[->] (3.2,5.5) edge (4,5.5);
\path[->] (1.5,2.8) edge (1.5,2);
\path[->] (4.5,2.8) edge (4.5,2);
\path[->] (1.5,4.2) edge (1.5,5); 
\path[->] (4.5,4.2) edge (4.5,5);
\draw[green,very thick] (3,3) -- (3,4);
\draw[red,very thick] (3,0) -- (3,3);
\draw[red,very thick] (3,4) -- (3,7);
\end{mosaicoPBC}
\label{vcd1}
}
\hspace{1cm}
\subfigure[Second inflation]{
\begin{mosaicoPBC}[scale=.3]
\Pl{6}\Br|{0}{6}\Pl{6}\Nl{6}
\Br-{6}{2}\Cr{0}{2}\Br-{6}{2}\Nl{2}
\Pl{6}\Br|{0}{6}\Pl{6}\Nl
\path[->] (5.5,3) edge (3.5,3);
\path[->] (5.5,11) edge (3.5,11);
\path[->] (6.5,3) edge (8.5,3);
\path[->] (6.5,11) edge (8.5,11);
\path[->] (3,5.5) edge (3,3.5);
\path[->] (9,5.5) edge (9,3.5);
\path[->] (3,8.5) edge (3,10.5);
\path[->] (9,8.5) edge (9,10.5);
\draw[green,very thick] (6,6) -- (6,8);
\draw[red,very thick] (6,0) -- (6,6);
\draw[red,very thick] (6,8) -- (6,14);
\end{mosaicoPBC}
\label{vcd2}
 } 
\caption{Virtual cross with degenerate arms}
\label{vcd}
\end{figure}

\paragraph*{\em {\bf Subcase 3.1)} The crosses have four exit points (including multiplicities).} Assume first that the crosses grow in the horizontal (on the left and right) and vertical (upward and downward) opposite directions, see Figure~\ref{vc4a}. Since the crosses cover the whole leaf $L$,  the continuous boundary $\partial_c \calr_\infty$ separates $L$ into $4$ connected components in the same way that 
decorations separate crosses. In general, the exit points located on the left and right sides of the crosses are positive and negative end points of the axes of horizontal arms pointing to the left and right. Similarly the exit points on the up and down sides are negative and positive end points of the axes of  vertical arms pointing upward and downward. Thus, if the crosses do not grow in all directions, there will be exit points on the right or left, up or down, remaining  in the interior of a square of constant side. But these points will be end points of the axes defining virtual arms pointing to the left or right, upward or downward. From the discussion of the previous case, it is clear that $\partial_c \calr_\infty$ still separates $L$ into $4$ connected components. On the other hand, we may have a degenerate virtual cross defined by a sequence of degenerate crosses of type $(a)$ or $(c)$ in Figure~\ref{deg}. In this case, there will always be $4$ virtual arms, while some will be degenerate, see Figure~\ref{vcd}.  In any event, as before, $\calr[\calt]$ separates into four $\calr_\infty$-equivalence classes. 
\medskip 

\paragraph*{\em {\bf Subcase 3.2)} The crosses have three exit points (including multiplicities).} This  situation is obviously very similar to the previous one, see Figure~\ref{vc3a}. Reasoning as before, we can see that $\partial_c \calr_\infty$ separates $L$ into $3$ connected components, and then $\calr[\calt]$ separates into three $\calr_\infty$-equivalence classes.
\medskip 

\begin{figure}
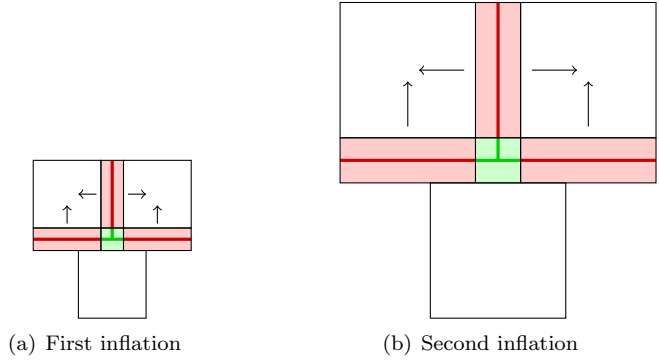
%[H]
\subfigure[First inflation]{
\begin{mosaicoPBC}[scale=.3]
\Pl{3}\Br|{1}{3}\Pl{3}\Nl{3}
\Br-{3}{1}\Cr{1}{1}\Br-{3}{1}\Nl{1}
\Va{2}\Pl{3}\Nl
\path[->] (2.8,1.5) edge (2,1.5);
%\path[->] (2.8,5.5) edge (2,5.5);
\path[->] (4.2,1.5) edge (5,1.5);
%\path[->] (4.2,5.5) edge (5,5.5);
\path[->] (1.5,2.8) edge (1.5,2);
\path[->] (5.5,2.8) edge (5.5,2);
%\path[->] (1.5,4.2) edge (1.5,5);
%\path[->] (5.5,4.2) edge (5.5,5);
\end{mosaicoPBC}
\label{vc3a1}
}
\hspace{1cm}
\subfigure[Second inflation]{
\begin{mosaicoPBC}[scale=.3]
\Pl{6}\Br|{2}{6}\Pl{6}\Nl{6}
\Br-{6}{2}\Cr{2}{2}\Br-{6}{2}\Nl{2}
\Va{4}\Pl{6}\Nl
\path[->] (5.5,3) edge (3.5,3);
%\path[->] (5.5,11) edge (3.5,11);
\path[->] (8.5,3) edge (10.5,3);
%\path[->] (8.5,11) edge (10.5,11);
\path[->] (3,5.5) edge (3,3.5);
\path[->] (11,5.5) edge (11,3.5);
%\path[->] (3,8.5) edge (3,10.5);
%\path[->] (11,8.5) edge (11,10.5);
\end{mosaicoPBC}
\label{vc3a2}
 } 
\caption{Virtual cross with $3$ exit points}
\label{vc3a}
\end{figure}

\begin{figure}
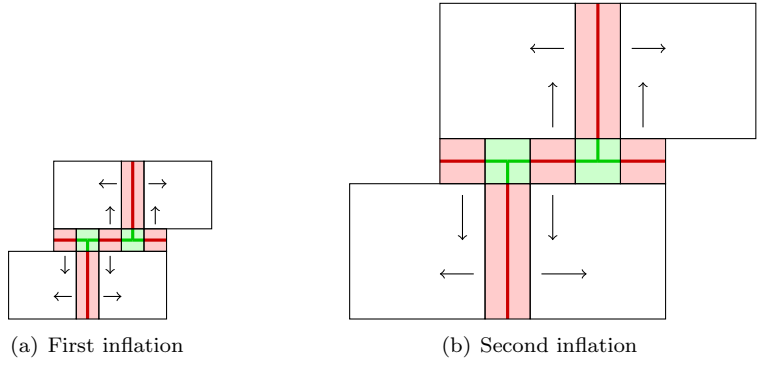
%[H]
\subfigure[First inflation]{\begin{mosaicoPBC}[scale=.3]
\Va{2}\Pl{3}\Br|{1}{3}\Pl{3}\Nl{3}
\Va{2}\Br-{1}{1}\Cr{1}{1}\Br-{1}{1}\Cr{1}{1}\Br-{1}{1}\Nl{1}
\Pl{3}\Br|{1}{3}\Pl{3}\Nl
\path[->] (4.8,1) edge (4,1);
\path[->] (2.8,6) edge (2,6);
\path[->] (6.2,1) edge (7,1);
\path[->] (4.2,6) edge (5,6);
\path[->] (4.5,2.8) edge (4.5,2);
\path[->] (6.5,2.8) edge (6.5,2);
\path[->] (2.5,4.2) edge (2.5,5);
\path[->] (4.5,4.2) edge (4.5,5);
\end{mosaicoPBC}
\label{twoc1}
}
\hspace{1cm}
\subfigure[Second inflation]{\begin{mosaicoPBC}[scale=.3]
\Va{4}\Pl{6}\Br|{2}{6}\Pl{6}\Nl{6}
\Va{4}\Br-{2}{2}\Cr{2}{2}\Br-{2}{2}\Cr{2}{2}\Br-{2}{2}\Nl{2}
\Pl{6}\Br|{2}{6}\Pl{6}\Nl
\path[->] (9.5,2) edge (8,2);
\path[->] (5.5,12) edge (4,12);
\path[->] (12.5,2) edge (14,2);
\path[->] (8.5,12) edge (10.5,12);
\path[->] (9,5.5) edge (9,3.5);
\path[->] (13,5.5) edge (13,3.5);
\path[->] (5,8.5) edge (5,10.5);
\path[->] (9,8.5) edge (9,10.5);
\end{mosaicoPBC}
\label{twoc2}
 } 
\caption{Two virtual crosses}
\label{twovc}
\end{figure}

Before dealing with the last case, we would like to point out that there may be different types of configurations in successive inflation steps. 
%In fact,  since $\pb_{n+1}$ is not necessarily inflated from $\pb_n$, there may be a virtual cross defined by a sequence of crosses whose number of exit points (including multiplicities) is not stationary.  
But because $\ppb_{n+1}$ is inflated from $\ppb_n$, the number of $\calr_\infty$-classes in each $\calr$-class $\calr[\calt]$ is determined by the configuration with the fewest number of exit points.
\medskip

\paragraph*{\em {\bf Case 4)} There is more than one virtual cross.} Let us recall that the side of the plaques of the partial decomposition $\calb_n$ tend to $\infty$. So the only possibility is that there are two virtual crosses defined by sequences of crosses with $3$ exit points (which may be nondegenerate or degenerate of type $(b)$ or $(d)$ as in Figure~\ref{deg}) connected by horizontal (or vertical) arms of bounded length and whose other horizontal (resp. vertical) arms grow in opposite directions. Furthermore, their only vertical (resp. horizontal) arms must also grow pointing in opposite directions, see Figure~\ref{twovc}. In this case, $\partial_c \calr_\infty$ still separates $L$ into $4$ connected components, and then $\calr[\calt]$ separates into four $\calr_\infty$-equivalence classes.
\end{proof}
\medskip 

Let us resume the previous results in the following statement:

\begin{proposition} \label{resboundary}
Let  $\X$ be the continuous hull of an aperiodic and repetitive planar tiling
%whose tiles are marked squares, 
and $X$ the total transversal defined by the choice of base points in the prototiles. Let $\calr$ be the EER on $X$ induced by the TDS of $\X$, and 
$\calr_\infty$ the minimal open AF equivalence subrelation defined by Robinson inflation. Then the boundary $\partial \calr_\infty$ is a $\calr_\infty$-thin meager closed subset of $X$ such that every $\calr$-equivalence class represented by an element of $\partial \calr_\infty$  separates into at most four $\calr_\infty$-equivalence classes. \qed
 \end{proposition}

\section{Preparing the absorption}  \label{filtrating}

In this section, we shall prepare the boundary to apply absorption and to complete the proof of the Affability Theorem in the next section. 
Let us start by considering the natural filtration
\[
\partial  \calr_\infty  =   \bigcup_{i=2}^4  \partial_i  \calr_\infty  
 =   \bigcup_{i=2}^4  \{ \, \calt \in X \tq \text{$\calr[\calt]$ is union of $i$ $\calr_\infty$-equivalence classes} \,  \}
\]
deduced from Proposition~\ref{resboundary}. We know that the whole boundary $\partial  \calr_\infty$ is a $\calr_\infty$-thin closed subset of $X$, but it is not $\calr_\infty$-\'etale. Given a partial transformation $\varphi : A \to B$ of $\calr_\infty$ between clopen subsets $A$ and $B$ of $X$ sending $\calt_1 \in \partial  \calr_\infty$ to $\calt_2 = \varphi(\calt_1) \in \partial  \calr_\infty$ (so that its graph is a clopen bisection of $\calr_\infty$ passing through $(\calt_1,\calt_2) \in \calr_\infty |_{\partial  \calr_\infty}$), it may be that there was a point $\hat{\calt}_1 \in A \cap \partial  \calr_\infty$  such that $\hat{\calt}_2 = \varphi(\hat{\calt}_1) \notin \partial  \calr_\infty$ (and hence $(\hat{\calt}_1,\hat{\calt}_2) \notin \calr_\infty |_{\partial  \calr_\infty}$).
On the other hand, although $\partial_3  \calr_\infty$ and $\partial_4  \calr_\infty$ are not closed, we can construct fundamental domains 
 $\calh_3$ for $\calr |_{\partial_3  \calr_\infty}$ and $\calh_4$ for $\calr |_{\partial_4  \calr_\infty}$.
 %, which are the union of countably many disjoint closed subsets of $\partial  \calr_\infty$, see Propositions~\ref{H3},~\ref{H40}~and~\ref{H4k}. 
They are naturally equipped with CEERs which are transverse to the restrictions of $\calr_\infty$. Before we continue, let us clarify this definition: 

\begin{definition} 
A subset $A$ of $X$ is said to be a \emph{fundamental domain} for the equivalence relation induced on its $\calr$-saturation when $A$ intersects all the $\calr$-equivalence classes of this saturation in exactly one point. Such a subset is obviously $\calr$-\'etale. 
\end{definition} 

Unfortunately, $\partial_2  \calr_\infty$ is not closed and it is not easy to construct a fundamental domain for  $\calr |_{\partial_2  \calr_\infty}$. As has been said, the idea is to replace $\calr_\infty$ with an equivalence subrelation $\hat{\calr}_\infty$ by splitting the  $\calr_n$-classes into smaller pieces so that $\partial \calr_\infty$ becomes $\hat{\calr}_\infty$-thin and $\hat{\calr}_\infty$-\'etale. Before this, we shall construct the announced fundamental domains $\calh_3$ and $\calh_4$. 
%They shall not be used directly, but they help to understand how to solve the problem.

\subsection{Constructing the fundamental domains $\calh_3$ and $\calh_4$}
For each $\calt \in \partial \calr_\infty$ and for each $n \in \N$, the equivalence class $\calr_n[\calt]$ is the intersection of $X$ with the $\calp''_n$-tile $\Pa''_n$ passing through the base point of $\calt$. From Proposition~\ref{isoperimetric2}, even if each $\calp''_n$-tile may have a very nasty boundary, it looks as a square on a large scale when $n \to \infty$. In Subsection~\ref{sSboundary}, we have denoted by $\Gamma^{(n)}$ the union of these boundaries, that is, $\Gamma^{(n)}$ is the union of the edges of the $\calp''_n$-tiling $\calt''_n$.  We shall distinguish the edges and the vertices of $\calt''_n$ (which define the total transversal $\check{X}^{(n)}$ described in Subsection~\ref{sSboundary}) 
from the edges and the vertices of the $n$-boundary $\Gamma^{(n)}$ (endowed with the graph structure derived from the original tiling $\calt$). 
%So the vertices of $\Pa''_n$ are $\calt''_n$-vertices (and hence they are vertices of $\Gamma^{(n)}$), but other $\calt''_n$-vertices may exist
%(and there are always many other vertices of $\Gamma^{(n)}$), and the sides of $\Pa''_n$ are paths in $\Gamma^{(n)}$ obtained from one or two $\calt''_n$-edges. 
%Somewhat abusively, we say that the vertices and the edges of $\calt''_n$ are the vertices and the edges of the $n$-boundary $\Gamma^{(n)}$. 
The intersection of the total transversal $\check{X}$ with the graph $\Gamma^{(n)}$ give us the whole vertex set of $\Gamma^{(n)}$
%as defined originally in Subsection~\ref{sSboundary}, 
whose degree function has been denoted 
$D_n : \check{X} \cap \Gamma^{(n)} \to \N$.
Let us now recall that  $\partial_c \calr_\infty$ is equipped with a treed equivalence relation induced by $\calf$ so that the class 
$\Gamma = \Gamma_\calt$ passing through $\calt$ coincides with the intersection of graphs $\Gamma^{(n)}$. According to  Proposition~\ref{finiteness}, this is a tree without terminal edges having $2$, $3$ or $4$ ends, see Figures~\ref{virtualboundaries23}~and~\ref{virtualboundaries4}.
%The persistence or not of their vertices and their type determine the different pieces in $\cpartial \calr_\infty$. 
\medskip

If we assume $\calt \in \cpartial_3 \calr_\infty$,  there is a virtual cross defined by a sequences of crosses $C_n$ with $3$ or $4$ arms for the intermediate tilings $\calt'_n$. For each $n \in \N$, there are $3$ or $4$ different  $\calp''_n$-tiles which meet in a neighborhood in $\Gamma^{(n)}$ of some fixed vertex $o = o_\Gamma$ in $C_n$. They form a $\calp''_n$-patch $\M_n(\Gamma)$.

\begin{definition}  We say that $o$ is the \emph{root} of $\Gamma$ and $\M_n(\Gamma)$ is a \emph{basic $\calp''_n$-patch around the root $o$}. We denote by $D(\M_n(\Gamma))$ the number of  $\calp''_n$-tiles of $\M_n(\Gamma)$. 
\end{definition} 

The decoration of $C_n$ and the axis of the corresponding arms $A_n$ form a rough model for $\Gamma$ in neighborhood of the root $o$, whereas the axis of the virtual arms provide rough models for the ends of $\Gamma$. 
%As we shall see in the next result, the number of $\calp''_n$-tiles around the root $o$ can be different from its degree $D( o)$.??

\begin{figure}
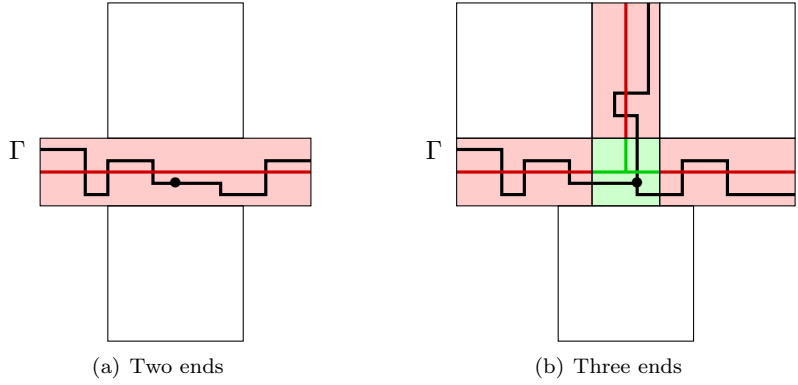

	\subfigure[Two ends]{
		\begin{mosaicoPBC}[scale=.3]
			\path[use as bounding box] (0,0) rectangle (12,15);
			
			\Va{3}\Pl{6}\Nl{6}
			\Br-{12}{3}\Nl{3}
			\Va{3}\Pl{6}

			\node at (-1,6.5) {$\Gamma$};
			%\node at (6,6.9) {$0$};
			\node at (6,8) {$\bullet$};
			
			\draw[very thick]
				(0,6.5) -- (2,6.5) -- (2,8.5) -- ( 3,8.5) -- ( 3,7) -- ( 5,7) --
				(5,8  ) -- (8,8  ) -- (8,8.5) -- (10,8.5) -- (10,7) -- (12,7);
		\end{mosaicoPBC}
		\label{vb2}
	}
	\hspace{1cm}
	\subfigure[Three ends]{
		\begin{mosaicoPBC}[scale=.15]
			\path[use as bounding box] (0,0) rectangle (30,30);
			
			\Pl{12}\Br|{6}{12}\Pl{12}\Nl{12}
			\Br-{12}{6}\Cr{6}{6}\Br-{12}{6}\Nl{6}
			\Va{9}\Pl{12}

			\node at (-2,13) {$\Gamma$};
			%\node at (8.6,7.9) {$0$};
			\node at (16,16) {$\bullet$};

			\draw[very thick]
				( 0,13) -- ( 4,13) -- ( 4,17) -- ( 6,17) -- ( 6,14) -- (10,14) --
				(10,16) -- (16,16) -- (16,17) -- (20,17) -- (20,14) -- (24,14) --
				(24,17) -- (30,17);
			\draw[very thick]
				(16,16) -- (16,10) -- (14,10) -- (14,8) -- (17,8) -- (17,0);
		\end{mosaicoPBC}
		\label{vb3}
	}
	\caption{Two and three ends}
	\label{virtualboundaries23}
\end{figure}

\begin{figure}
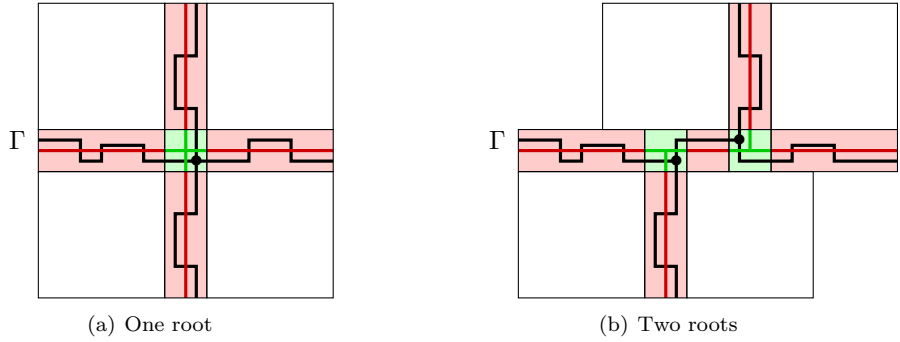

	\subfigure[One root]{
	\begin{mosaicoPBC}[scale=.14]
		\Pl{12}\Br|{4}{12}\Pl{12}\Nl{12}
		\Br-{12}{4}\Cr{4}{4}\Br-{12}{4}\Nl{4}
		\Pl{12}\Br|{4}{12}\Pl{12}\Nl

		\node at (-2,13) {$\Gamma$};
		%\node at (8.6,7.9) {$0$};
		\node at (15,15) {$\bullet$};

		%rama horizontal
		\draw[very thick] (0,13) -- (4,13) -- (4,15) -- (6,15) -- (6,13.5) -- (10,13.5) -- (10,15) -- 
		%\draw[very thick]  (15,15) -- (15,13);
		%\draw[very thick]  (15,13) -- (21,13);
		%\draw[very thick]  (21,13) -- (21,15);
		(15,15) -- (20,15) -- (20,13) -- (24,13) -- (24,15) -- (28,15);
		
		%primera rama vertical 
		\draw[very thick]  (15,15) -- (15,20) -- (13,20) -- (13,25) -- (15,25) -- (15,28);

		%segunda rama vertical
		\draw[very thick]  (15,15) -- (15,10) -- (13,10) -- (13,5) -- (15,5) -- (15,0);
	\end{mosaicoPBC}
	\label{vb4}
	}
	\hspace{1cm}
	\subfigure[Two roots ]{	 
	\begin{mosaicoPBC}[scale=.14]
		\Va{8}\Pl{12}\Br|{4}{12}\Pl{12}\Nl{12}
		\Br-{12}{4}\Cr{4}{4}\Br-{4}{4}\Cr{4}{4}\Br-{12}{4}\Nl{4}
		\Pl{12}\Br|{4}{12}\Pl{12}\Nl

		\node at (-2,13) {$\Gamma$};
		%\node at (8.6,7.9) {$0$};
		\node at (15,15) {$\bullet$};
		\node at (21,13) {$\bullet$};
		%rama horizontal
		\draw[very thick] (0,13) -- (4,13) -- (4,15) -- (6,15) -- (6,13.5) -- (10,13.5) -- (10,15) -- (15,15) -- (15,13) -- (21,13) -- (21,15) -- (26,15) -- (26,13.5) -- (30,13.5) -- (30,15) -- (36,15);
		%primera rama vertical 
		\draw[very thick]  (15,15) -- (15,20) -- (13,20) -- (13,25) -- (15,25) -- (15,28);
		%segunda rama vertical
		\draw[very thick]  (21,13) -- (21,10) -- (23,10) -- (23,5) -- (21,5) -- (21,0);
	\end{mosaicoPBC}
	\label{vb4k}
	}
	\caption{Four ends}
	\label{virtualboundaries4}
\end{figure}

\begin{proposition}
	\label{H3}
	The natural EER induced on the $\calr$-saturation of $\cpartial_3  \calr_\infty$ admits 
	a fundamental domain $\check{\calh}_3$ which is the union of  countably many disjoint
	closed subsets $\check{\calh}_{3,m}$. 
\end{proposition}

\begin{proof}
	Given $\calt \in \cpartial_3  \calr_\infty$, let $\Gamma = \Gamma_\calt$ be the trace 
	 of the continuous boundary $\partial_c \calr_\infty$ on the leaf $L = L_\calt$. According to the 
	 previous discussion (see again the proof of Proposition~\ref{finiteness}), we know that 
	 $\M_n(\Gamma)$  grows in at least  three different directions, see Figure~\ref{vb3}.  But notice that 
	 we may also find a basic patch of  this kind in a leaf of the saturation of 
	 $\calt \in \cpartial_4  \calr_\infty$ having exactly two virtual crosses. 
	 In this case, $\Gamma$ has two roots $o$ and $o'$ of degree $D(o)=D(o')=3$, see Figure~\ref{vb4k}. 
	 If we assume that $\M_n(\Gamma)$ contains a ball of radius $n$ centered at $o$ 
	 and if $n$ is larger than the distance $\ell$ between $o$ and $o'$, we need $4$ different 
	 $\calp''_n$-tiles to cover the ball. In other words, for a large enough $n$, the basic 
	$\calp''_n$-patches around $o$ and $o'$ have nontrivial intersection, and we 
	can replace each of the original basic $\calp''_n$-patches (made up of $3$ different 
	$\calp''_n$-tiles) with its union (made up of $4$ different 
	$\calp''_n$-tiles). We still denote by $\M_n(\Gamma)$ the new basic 
	$\calp''_n$-patch around $o$ and $o'$. Thus $D(\M_n(\Gamma)) = 4$ is different from the degrees $D(o)=D(o')=3$.
	Now, since $o$ belongs to $\cpartial_3  \calr_\infty$, there is a minimal  integer
	 $m(\Gamma) \geq 0$ such that $D(\M_n(\Gamma)) = 3 = D(o)$ for all $n \geq m(\Gamma)$. 
	  \medskip

	 Let $X_n(\Gamma)$ be the set of tilings in $\check{X}$ containing the patch $\M_n(\Gamma)$ 
	 (which  cover the ball of radius $n$) around the origin. This is a clopen subset of $\check{X}$. 
	 Let us note that  each clopen subset $X_n(\Gamma) \cap \cpartial \calr_\infty$ is a fundamental 
	 domain for the equivalence relation induced on  $\partial_c \calr_\infty$. 
%	 Indeed, even if $X_n(\Gamma) \cap \cpartial \calr_\infty$ meets
%	a leaf (which belong to the saturation of $\calt \in \cpartial_4  \calr_\infty$) having 
%	exactly two virtual crosses, it meets just once, see again Figure~\ref{vb4k}. 
	But $X_n(\Gamma)$ is never a fundamental domain for  the EER induced on the $\calr$-saturation of 
	$\cpartial_3 \calr_\infty$. However, the intersection
	$$
	\check{\mathcal{H}}_{3,\Gamma} = \bigcap_{n \geq m(\Gamma)} X_n(\Gamma)
	$$ 
	is a closed subset of $\cpartial_3  \calr_\infty$ which meets all the $\calr$-equivalence classes 
	in the $\calr$-saturation of $\cpartial_3  \calr_\infty$ at most one time, although its $\calr$-saturation 
	may be smaller than that of $\cpartial_3  \calr_\infty$. Unfortunately, there are uncountably many finite 
	labeled trees $\Gamma$ with $3$ ends. However, we can consider the closed set
	$$
	 \check{\mathcal{H}}_{3,\Gamma_{m}}= 
	 \bigcap_{n \geq m} \, \bigcup_{\Gamma'_m = \Gamma_m} X_n(\Gamma')
	 $$ 
	where $\Gamma'$ represents any tree in $\partial_c \calr_\infty$ such that 
	$\M_m(\Gamma') =  \M_m(\Gamma)$ and therefore 
	$\Gamma'_m= \Gamma' \cap \M_m(\Gamma') = \Gamma \cap \M_m(\Gamma) = \Gamma_m$. 
	For each $m \in \N$, there are only finitely many trees $\Gamma_m$, and then the union 
	$\check{\mathcal{H}}_{3,m}$ of the closed sets $\check{\mathcal{H}}_{3,\Gamma_{m}}$ still is closed. The union 
	$$
	\check{\calh}_3 = \bigcup_{m \in \N} \check{\calh}_{3,m}
	$$
	 is fundamental domain which meets each leaf of the saturation of $\cpartial_3 \calr_\infty$ in a 
	 unique point: the root $o$ of the boundary $\Gamma$. 
	 \end{proof}

%\begin{remark} \label{remarkaction3}
%%i) In the proof above, we have actually constructed a continuous map $D$ from the subset of roots in $\cpartial_{\geq 3}  \calr_\infty = \cpartial_3  \calr_\infty \cup \cpartial_4  \calr_\infty$ to the Cantor set $\{3,4\}^\N$. This subset can be described as the inverse limit 
%%$\varprojlim \mathfrak{M}_n$ where $\mathfrak{M}_n$ is the set of the discrete basic patches 
%%$M_n = \M_n \cap \cpartial  \calr_\infty$ around any vertex of degree $3$ or $4$ and the obvious onto map $\mathfrak{M}_n \to \mathfrak{M}_{n-1}$ is given by the natural division of the inflated patches. 
%%Then $\check{\mathcal{H}}_3$ is the inverse image by $D$ of the set of all sequences $s_n$ which are cofinal to the constant sequence $3$.
%%\medskip 
%%
%%\noindent
%Going back to the original total transversal $X$, we obtain $3$ disjoint fundamental domains for the $\calr$-saturation of $\partial_3  \calr_\infty$. %{\blue They are related by partial transformations of $\calr$ that generate a free $\Z_3$-action $\Psi_3$ on their their disjoint union.} 
%We  shall denote by $\calh_3$ any of these fundamental domains.
%\end{remark} 

The construction of a fundamental domain for the $\calr$-saturation of $\cpartial_4  \calr_\infty$ should be a little different since many leaves admit two roots. Only in the case when the leaves have a unique real root of degree $4$, see Figure~\ref{vb4}, we can conclude as in the previous case. The rest of four-divided leaves (with two roots of degree $3$ as in Figure~\ref{vb4k}) should be treated in another way.

\begin{proposition} \label{H40}
There is a closed subset $\check{\calh}_{4,0} \subset \cpartial_4  \calr_\infty$ of the continuous boundary $\cpartial  \calr_\infty$ intersecting all $\calr$-equivalence classes in at most one point: the unique vertex of degree $4$ of the corresponding tree in $\partial_c  \calr_\infty$. \qed
\end{proposition}

%\begin{remark} \label{remarkaction4}
%As above, using this lemma, we obtain $4$ disjoint closed fundamental domains for the $\calr$-saturation of $\partial_4  \calr_\infty$. {\blue In this case, they are related by partial transformations of $\calr$ that generate a free $\Z_4$-action $\Psi_4$ on their disjoint union.} Any of these fundamental domains will be denoted by  $\calh_{4,0}$.
%\end{remark}

In general, for each tiling $\calt \in \cpartial_4  \calr_\infty - \check{\calh}_{4,0}$, we have pairs of crosses $C_n$ and $C_n^\prime$ with $3$ exit points defining two virtual crosses connected by a virtual arm of bounded length. Let us assume that these crosses are of the type described in Figure~\ref{vb4k}, that is, the corresponding arms are included in the union of a horizontal ribbon and two vertical semi-ribbons pointing upward and downward. Like for tilings in $\cpartial_3  \calr_\infty$, the union of the decoration of $C_n$ and $C_n^\prime$ and the axis of the corresponding arms $A_n$ and $A'_n$ forms a rough model for the tree $\Gamma = \Gamma_\calt$ passing through $\calt$ in a neighborhood of the two roots $o$ and $o'$, see Figure~\ref{vb4k}. In this case, the axes of the vertical semi-ribbons meet the horizontal axis in two different points, making up a rough  model for $\Gamma$. We define $\check{\calh}_{4,\ell}$ as the set of tilings $\calt \in \cpartial_4  \calr_\infty$ with origin in a $o$ such that the distance to the other root  $o'$ is equal to a nonnegative integer $\ell$.

\begin{proposition} \label{H4k}
	The natural EER induced on the $\calr$-saturation of $\cpartial_4  \calr_\infty$ admits 
	a fundamental domain $\check{\calh}_4$ which is the union of  countably many disjoint
	closed subsets $\check{\calh}_{4,\ell}$. 
\end{proposition}

\begin{proof} 
	We  start by  fixing a positive integer $\ell \geq 1$, and considering a tiling  
	$\calt \in \cpartial_4  \calr_\infty$ such that there is a geodesic path $\gamma$ of length $\ell$ joining 
	$o$ and $o'$. If $n \in \N$ is large enough, 
	there are $4$ different $\calp''_n$-tiles that meet together in a neighborhood of $\gamma$ in $\Gamma$. 
	As in the proof of Lemma~\ref{H3}, they form a basic $\calp''_n$-patch
	around any vertex of $\gamma$ which coincides with $\M_n(\Gamma)$. 
	%In the same way, 
	 We denote by $X_n(\Gamma)$ the set of tilings in $\check{X}$ having this patch around the 
	 origin (that becomes one of the vertices of $\gamma$). In this case, since all the points of 
	 $\gamma$ belong to $\cpartial_4  \calr_\infty$, there is a  minimal integer $m(\gamma) \geq 0$ 
	 such that $D(\M_n(\Gamma)) = 4$ for all  $n \geq m(\gamma)$ whereas $D(o)=D(o') = 3$. 
	 Therefore, if $\Gamma'$ represents any tree 
	 in $\partial_c \calr_\infty$ such that $\M_m(\Gamma') = \M_m(\Gamma)$, then 
	$$
	\check{\mathcal{G}}_{4,\gamma} = \bigcap_{n \geq m}  \, 
	\bigcup_{\Gamma'_m= \Gamma_m} X_n(\Gamma')
	 $$ 
	is a closed subset of $\cpartial_4  \calr_\infty$ which meets all the $\calr$-equivalence classes of 
	the $\calr$-saturation of $\cpartial_4  \calr_\infty$ in at most $\ell+1$ points. Since there is a 
	finite number of paths of length $\leq \ell$ starting from the origin, the union 
	$\check{\mathcal{G}}_{4,\ell}$ of all these closed 
	subsets $\check{\mathcal{G}}_{4,\gamma}$ of $\cpartial_4  \calr_\infty$ still is closed. Moreover, each closed subset 
	$\check{\mathcal{G}}_{4,\gamma}$ split into $\ell +1$ closed subsets intersecting all 
	$\calr$-equivalence classes in at most one point. We denote by 
	$\check{\calh}_{4,\gamma}$ any of these closed sets. As before, their union $\check{\calh}_{4,\ell}$ 
	still has the same property. Then 
	$$
	\check{\calh}_4 = \bigcup_{\ell \geq 0} \check{\calh}_{4,\ell}
	$$
	is a fundamental domain for the $\calr$-saturation of $\cpartial_4  \calr_\infty$. 
	\end{proof}

%\begin{remark}
%As before,  $\check{\calh}_4$ provides $4$ disjoint fundamental domains for the $\calr$-saturation of  $\partial_4  \calr_\infty$. They are related by partial transformations of $\calr$, and any of these fundamental domains (which will be denoted by $\calh_4$) split into countably many closed subsets $\calh_{4,\ell}$.
%\end{remark}

Although Propositions~\ref{H3}~and~\ref{H4k} has been proved in a similar way, there is an important difference between $\check{\calh}_3$ 
%(which will become similar to $\check{\calh}_{4,0}$) 
and $\check{\calh}_{4,\ell}$: 

\begin{proposition} \label{H3closed}
The fundamental domain $\check{\calh}_3$ is a closed subset of $\check{X}$.
\end{proposition}

\begin{proof} Assume that $\calt_n$ is a sequence of tilings belonging to $\check{\calh}_3$ that converges to a tiling $\calt$. According to  Proposition~\ref{semi-continuous}, the set $\{ \calt \in \partial \calr_\infty | D(\calt) \geq 3 \} \supset \check{\calh}_3$ is a closed subset of 
$\cpartial \calr_\infty \subset \check{X}$. Thus, the origin of $\calt$ is placed at a root of degree $3$ or $4$. In the first case, if $\calt \notin \check{\calh}_3$, 
there is a $\calt$-patch (consisting of four $\calp''_n$-tiles) containing a second root of the boundary $\Gamma = \Gamma_\calt$. For a large enough $n$, there are $\calt_n$-patches (consisting of four $\calp''_n$-plaques in the same flow boxes in $\ppb_n$) which determine two different roots of 
the boundaries $\Gamma_{\calt_n}$ of $\calt_n$. But this contradicts the fact that $\calt_n \in \check{\calh}_3$. In the second case, if 
$D(\calt) = 4$, we can argue similarly to conclude that $D(\calt_n) = 4$ for a large enough $n$, contradicting again the fact that $\calt_n \in \check{\calh}_3$. 
\end{proof} 

\begin{remark} \label{remarkfd}
Going back to the total transversal $X$, we can replace $\check{\calh}_3$ and $\check{\calh}_{4,\ell}$ by $3$ or $4$ disjoint closed fundamental domains for $\calr$. We shall denote by $\calh_3$ and $\calh_{4,\ell}$ any of these closed fundamental domains.
\end{remark} 

\subsection{Making a patchwork} As we have already said, it is not clear how a fundamental domain for $\partial_2  \calr_\infty$ can be constructed. Now, we shall replace $\calr_n$ with an equivalence subrelation $\hat{\calr}_n$ obtained using the following procedure. 
\medskip 

Firstly, we color each $\calt''_n$-edge of the tile $\Pa''_n$ and consequently all the $\calp''_{n-1}$-tiles touching $\partial \Pa''_n$. Thus, a well-defined color $c \in \calc = \{ 1, \dots , 8 \}$ is associated with any $\calp''_{n-1}$-tile touching only the interior of the $\calt''_n$-edges in $\Pa''_n$, while two colors are necessary to encode each $\calp''_{n-1}$-tile incident to some $\calt''_n$-vertex in $\Pa''_n$. Secondly, we split the complementary of their union into the same number of $\calp''_{n-1}$-patches $\leq 8$, and finally we color arbitrarily these $\calp''_{n-1}$-patches with the same number of colors, see Figure~\ref{calr'n.a}. Now, if the origin of the tiling $\calt \in X$ belongs to $\Pa''_n$, the class $\hat{\calr}_n[\calt]$ is defined as the intersection of $X$ with each $\calp''_{n-1}$-tile having the same color that the $\calp$-tile containing the origin. 
\medskip 

Note, however, that the base point of each $\calp$-tile incident to some $\calt''_n$-vertex must be doubled by attaching two color codes at each stage of the inflation. In some cases, namely when there are two $\calp$-tiles included in the same $\calp''_n$-tile, the color of each incident $\calt''_n$-edge determine a single color code for each base point, see Figure~\ref{calr'n.b}. In other word, we can assume that the two color codes are equal, so the corresponding colored base points are identified by $\hat{\calr}_n$. Otherwise, when there is a unique $\calp$-tile included in the $\calp''_n$-tile, each of two color codes attached to the base point is determined by the color of each $\calt''_n$-edge incident with the $\calt''_n$-vertex.  In other words, if we consider the $n$-boundary $\Gamma^{(n)} \subset \check{X}^{(n)} \subset \check{X}$, each of four base points in $X^{(n)} \subset X$ associated with a vertex of $\calt''_n$ is doubled by attaching two color codes in $\calc$, which may be different or not, while each of four base points associated with the other vertices of $\Gamma^{(n)}$ has a well-defined color. 
\medskip 

\begin{figure}
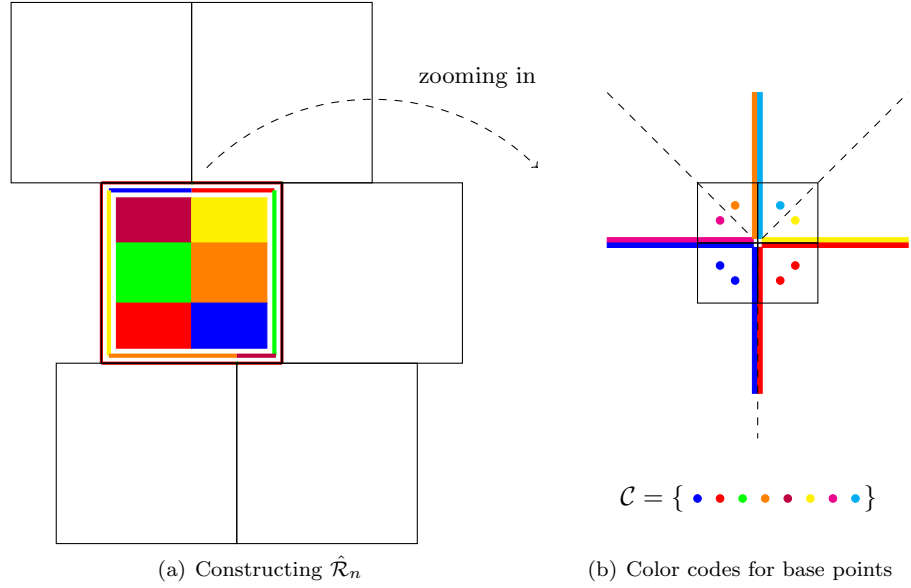

	\subfigure[Constructing $\hat{\calr}_n$]{
\begin{mosaicoPBC}[scale=.2]
\Pl{12}\Pl{12}\Nl{12}
\Va{6}\Br-{6}{0}\Cr{0}{0}\Br-{6}{0}\Cr{0}{0}\Nl{0}
\Va{6}\Br|{0}{12}\Pl{12}\Br|{0}{12}\Pl{12}\Nl{12}
\Va{6}\Br-{9}{0}\Cr{0}{0}\Br-{3}{0}\Nl{0}
\Va{3}\Pl{12}\Pl{12}
\draw[blue,ultra thick] (6.5,12.5) -- (12,12.5);
\draw[red,ultra thick] (12,12.5) -- (17.5,12.5);
\draw[yellow,ultra thick] (6.5,12.5) -- (6.5,23.5);
\draw[green,ultra thick] (17.5,12.5) -- (17.5,23.5);
\draw[orange,ultra thick] (6.5,23.5) -- (15,23.5);
\draw[purple,ultra thick] (15,23.5) -- (17.6,23.5);
\draw [fill=purple,purple] (7,16) rectangle (12,13);
\draw [fill=yellow, yellow] (12,16) rectangle (17,13);
\draw [fill=green,green] (7,20) rectangle (12,16);
\draw [fill=orange, orange] (12,20) rectangle (17,16);
\draw [fill=red,red] (7,23) rectangle (12,20);
\draw [fill=blue, blue] (12,23) rectangle (17,20);
%\path[.>] (13,11) edge [out=-45, in=-135] (32,11);
\draw [dashed,->] (13,11)  to[out=-45, in=-135] (35,11);
\node at (31,5){\small zooming in};
\end{mosaicoPBC}
\label{calr'n.a}
}
	\subfigure[Color codes for base points]{
\begin{mosaicoPBC}[scale=.2]
\Va{10}\Nl{6}
\Va{6}\Pl{4}\Pl{4}\Nl{4}
\Va{6}\Pl{4}\Pl{4}\Nl{4}
\draw[transparent] (0,0) -- (20,30);
\node at (3,27) {$\calc = \{$};
\draw[blue,fill=blue] (6,27) circle (0.25);
\draw[red,fill=red] (7.5,27) circle (0.25);
\draw[green,fill=green] (9,27) circle (0.25);
\draw[orange,fill=orange] (10.5,27) circle (0.25);
\draw[purple,fill=purple] (12,27) circle (0.25);
\draw[yellow,fill=yellow] (13.5,27) circle (0.25);
\draw[magenta,fill=magenta] (15,27) circle (0.25);
\draw[cyan,fill=cyan] (16.5,27) circle (0.25);
\node at (17.5,27) {$\}$};
\draw[blue,fill=blue] (7.5,11.5) circle (0.25);
\draw[blue,fill=blue] (8.5,12.5) circle (0.25);
\draw[red,fill=red] (12.5,11.5) circle (0.25);
\draw[red,fill=red] (11.5,12.5) circle (0.25);
\draw[magenta,fill=magenta] (7.5,8.5) circle (0.25);
\draw[orange,fill=orange] (8.5,7.5) circle (0.25);
\draw[cyan,fill=cyan] (11.5,7.5) circle (0.25);
\draw[yellow,fill=yellow] (12.5,8.5) circle (0.25);
\draw[magenta,fill=magenta] (0,9.65) rectangle (9.7,10);
\draw[orange,fill=orange] (9.65,0) rectangle (10,9.7);
\draw[cyan,fill=cyan] (10,0) rectangle (10.3,9.7);
\draw[yellow,fill=yellow] (10.3,9.65) rectangle (20,10);
\draw[blue, fill=blue] (0,10) rectangle (9.7,10.3);
\draw[blue, fill=blue] (9.65,10.3) rectangle (10,20);
\draw[red,fill=red] (10,10.3) rectangle (10.3,20);
\draw[red,fill=red] (10.3,10) rectangle (20,10.3);
\draw [dashed] (0,0)  -- (9.7,9.7);
\draw [dashed] (10.3,9.7)  -- (20,0);
\draw [dashed] (10,10)  -- (10,23);

\end{mosaicoPBC}
\label{calr'n.b}
}
\caption{Constructing $\hat{\calr}_n$ and coloring $\calp$-tiles}
\label{calr'n}
\end{figure}

In order to assure that $\hat{\calr}_n$ is a CEER, it is enough to consider the box decomposition $\ppb_n$ inflated from $\ppb_{n-1}$ provided by Theorem~\ref{Robinsoninflation}, and coloring in the same way all the plaques  $\Pa''_n \times \{ \ast \} $ contained in the same flow box 
$\B''_n \cong \Pa''_n \times X''_n$. 
Now, assuming that each colored $\calp''_{n-1}$-patch contains a ball of the same radius that goes to infinity as $n \to +\infty$, we have immediately the following version of Proposition~\ref{minimal}:

\begin{proposition}
	\label{minimal'}
	The inductive limit $\hat{\calr}_\infty= \varinjlim \hat{\calr}_n$ is a minimal open
	AF equivalence subrelation of  $\calr_\infty$ whose boundary $\partial \hat{\calr}_\infty \supset \partial \calr_\infty$. \qed
\end{proposition}

As $X$ and $X^{(n)}$, the boundary $\partial \calr_\infty$ becomes bigger than the original one. If $\calt \in \partial _3 \calr_\infty \cup \partial_4 \calr_\infty$, each of four base points in $\partial \calr_\infty$ associated with the roots (one or two) of $\Gamma = \Gamma_\calt$ is doubled and encoded with two color codes in $\calc$ at each stage of the inflation. If the root $o$ has degree $D(o) = 3$, there are exactly two pairs of colored base points whose color codes coincide from a certain inflation stage (so the associated colored base points are $\hat{\calr}_\infty$-related), but it is not possible when $D(0)=4$, see again Figure~\ref{calr'n.b}. Now, we can prove the following fundamental result: 

\begin{proposition} \label{R'etale}
The boundary $\partial \calr_\infty$ is $\hat{\calr}_\infty$-\'etale. 
\end{proposition} 

\begin{proof} Assume that $(\calt_1,\calt_2) \in \hat{\calr}_\infty$ with $\calt_1,\calt_2 \in \partial  \calr_\infty$. We need to construct a partial transformation 
$\varphi : A \to B$ of $\hat{\calr}_\infty$ between open neighborhoods $A$ of $\calt_1$  and $B$ of $\calt_2$ in $X$ such that 
$\calt_2 = \varphi(\calt_1)$ and $\hat{\calt}_1 \in A \cap \partial \calr_\infty$ if and only if $\hat{\calt}_2  = \varphi(\hat{\calt}_1) \in B \cap \partial \calr_\infty$. By definition, there is $n \in \N$ such that  $(\calt_1,\calt_2) \in \hat{\calr}_n$. It follows that $\calt_1$ and $\calt_2$ belong to a single plaque of a multicolored flow box $\B''_n$. This means that its origins belong to the same $\calp''_n$-tile. 
Since $\calt_1,\calt_2 \in \partial  \calr_\infty$, these points belong to two $\calp''_{n-1}$-tiles that meet $\partial \Pa''_n$ along the same colored $\calt''_n$-edge. By triviality of $\B''_n \cong \Pa''_n \times X''_n$, we have a partial transformation $\varphi : A \to B$ of $\calr$ between two copies $A$ and $B$ of $X''_n$ passing through $\calt_1$ and $\calt_2$ respectively such that $\calt_2 = \varphi(\calt_1)$. By construction, since the base points determined by 
$\calt_1$ and $\calt_2$ have the same color code, the graph of $\varphi$ is an open subset of the graph of $\hat{\calr}_n$ and therefore of the graph of $\hat{\calr}_\infty$. Finally, if a tiling $\hat{\calt}_1 \in A \cap \partial \calr_\infty$, then the tiling $\hat{\calt}_2  = \varphi(\hat{\calt}_1) \in B \cap \partial \calr_\infty$ because they are in the same multicolored plaque in $\B''_n$ and the $\calp''_{n-1}$-tiles containing them must meet the same colored 
$\calt''_n$-edge. Indeed, if this $\calt''_n$-edge belong to $\partial \calr_\infty$, then  $\hat{\calt}_1$ and $\hat{\calt}_2$ are  simultaneously  in $ \partial \calr_\infty$, while none of these tilings are in $\partial \calr_\infty$ if the $\calt''_n$-edge does not belong to $\partial \calr_\infty$.
\end{proof}

%The next step consist of extending Proposition~\ref{thin} to this new situation: 

\begin{proposition} \label{R'thin}
The boundary $\partial \calr_\infty$ is $\hat{\calr}_\infty$-thin. 
\end{proposition}

\begin{proof} The proof reduces to adapt the proof of Proposition~\ref{thin}. Firstly, notice that any discrete flow box $B''_n \cong P''_n \times X''_n$ defined from an element of $\ppb_n$ split into a finite number $\leq 8$ of flow boxes $B''_n(c) \cong P''_n (c) \times X''_n$ whose plaques have the same color encoded by $c \in \calc = \{ 1, \dots , 8 \}$. So they are $\hat{\calr}_n$-classes contained in one single $\calr_n$-class. For any $\hat{\calr}_\infty$-invariant probability measure $\mu$ on $X$, we can similarly argue to obtain:
\begin{align*}
		\mu(\partial \calr_n)  = \sum_{\B''_n \in \ppb_n} \mu(\partial_v B''_n) 
		& = \sum_{\B''_n \in \ppb_n} \sum_{c \in \calc}  \card \partial P''_n(c) \mu(X''_n) \\
		& = \sum_{\B''_n \in \ppb_n} \sum_{c \in \calc}  \frac{\card \partial P''_n(c)}{\card P''_n(c)} \mu(B''_n) \\
		& \leq  \max_{\Pa''_n \in \calp''_n}  \Big\{ \frac{\card \partial P''_n(c)}{\card P''_n(c)} \Big\} 
			\sum_{\B''_n \in \ppb_n} \mu(B''_n)  \\
		& =  \max_{\Pa''_n \in \calp''_n} \Big\{\frac{\card \partial P''_n(c)}{\card P''_n(c)} \Big\} \, \mu(X) 
			\leq 8  \max_{\Pa''_n \in \calp''_n} \Big\{\frac{\card \partial P''_n}{\card P''_n(c)} \Big\} 
	\end{align*}
where $\card P''_n(c)$ and $\card \partial P''_n(c)$ are the number of elements of the discrete colored piece $P''_n(c)$ and its boundary 
$\partial P''_n(c)$. Since $\card P''_n(c)$ and $\card P''_n$ have the same growth type, the isoperimetric ratio $\frac{\card \partial P''_n}{\card P''_n(c)}$ still converges to $0$ and hence $\partial \calr_\infty$ is $\hat{\calr}_\infty$-thin. 
\end{proof}

\section{Absorbing the boundary} \label{absorbing}

The aim of this last section is to prove Theorem~\ref{affabilitytheorem}. We start by constructing a CEER $\calk_2$ on $Y_2 = \partial \calr_\infty$ that is transverse to $\hat{\calr}_\infty \! \mid_{Y_2}$. By applying the Absorption Theorem of \cite{GMPS2}, see Theorem~\ref{thabsorption}, we shall obtain that $\hat{\calr}_\infty \vee \calk_2$ is a minimal AF equivalence relation OE to $\hat{\calr}_\infty$. 
In the next step, we construct a new CEER $\calk_3$ that is generated by two compact \'etale equivalence subrelations which are transverse to $\hat{\calr}_\infty \vee \calk_2$ in restriction to some \'etale and thin closed subset $Y_3$ of $Y_2$, and hence 
$\hat{\calr} _\infty \vee \calk_2 \vee \calk_3$ is another minimal AF equivalence relation OE to $\hat{\calr}_\infty$. In the last step, we shall complete the proof of Theorem~\ref{affabilitytheorem} by constructing a new CEER $\calk_4$, but now $\calk_4$ is generated by the union of two compact \'etale equivalence subrelations which are transverse to $\hat{\calr}_\infty \vee \calk_2 \vee \calk_3$ in restriction to the union $Y_4$ of an increasing sequence of \'etale and thin closed subsets of $Y_3$. We shall deduce that 
$$
\calr = \hat{\calr} _\infty \vee \calk_2 \vee \calk_3 \vee \calk_4
$$
is affable. 

\subsection{Constructing $\calk_2$} According to Proposition~\ref{R'etale}, the boundary $\partial \calr_\infty$ is $\hat{\calr}_\infty$-\'etale. 
For each pair $(\calt_1,\calt_2) \in \hat{\calr}_\infty$ with $\calt_1,\calt_2 \in \partial  \calr_\infty$, we have constructed a partial transformation $\varphi : A \to B$ of $\hat{\calr}_\infty$ such that $\calt_2 = \varphi(\calt_1)$ and $\hat{\calt}_1 \in A \cap \partial \calr_\infty$ if and only if $\hat{\calt}_2  = \varphi(\hat{\calt}_1) \in B \cap \partial \calr_\infty$. By definition, $\calt_1$ and $\calt_2$ belong to $\ppb_{n-1}$-plaques in the same multicolored $\ppb_n$-plaque $\Pa''_n$ that meet $\partial \Pa''_n$ along the same colored edge.
Then $A$ and $B$ are local transversals for $\ppb_{n-1}$ that meet these $\ppb_{n-1}$-plaques. On the other hand,  given $\calt_1 \in \partial  \calr_\infty$, there is a unique $\calt^\pitchfork_1 \in \partial \calr_\infty$ such that $\calt_1$ and $\calt^\pitchfork_1$ belong to $\ppb_{n-1}$-plaques contained in two different $\ppb_n$-plaques, so 
$(\calt_1,\calt^\pitchfork_1) \notin \calr_\infty$. To each of these $\ppb_{n-1}$-plaques, we have associated a well-defined color (not necessary equal) even if they are incident to some vertex at some inflation stage, 
see Figure~\ref{calr'n}. Respecting these colors and arguing as in the proof of Proposition~\ref{R'etale}, we construct a partial transformation $\varphi^\pitchfork: A^\pitchfork \to B^\pitchfork$ of $\calr$ such that $\calt^\pitchfork_1 = \varphi^\pitchfork(\calt_1)$ and  $\hat{\calt}_1 \in A^\pitchfork \cap \partial \calr_\infty$ if and only if $\hat{\calt}^\pitchfork_1  = \varphi^\pitchfork(\hat{\calt}_1) \in B^\pitchfork \cap \partial \calr_\infty$. Here $A^\pitchfork$ and $B^\pitchfork$ are clopen subsets of two local transversals relative to the box decomposition $\ppb_n$.

\begin{definition}
We say that $\calt_1$ is {\em $\calk_2$-equivalent} to $\calt^\pitchfork_1$, and so we have a finite equivalence relation $\calk_2$ on 
$\partial  \calr_\infty$. 
\end{definition} 

\begin{proposition} \label{K_2}
The equivalence relation $\calk_2$ is a CEER transverse to $\hat{\calr}_\infty \! \mid_{\partial \calr_\infty}$.
\end{proposition} 

\begin{proof} By construction, for each pair $(\calt_1,\calt^\pitchfork_1) \in \calk_2$, the graph of the local transformation 
$\varphi^\pitchfork |_{\partial \calr_\infty}: A^\pitchfork  \cap \partial \calr_\infty \to B^\pitchfork \cap \partial \calr_\infty$
becomes a bisection of $\calk_2$ containing $(\hat{\calt}_1,\hat{\calt}^\pitchfork_1 ) = (\hat{\calt}_1,\varphi^\pitchfork(\hat{\calt}_1))$ for all 
$\hat{\calt}_1 \in \hat{A} \cap \partial \calr_\infty$. Then $\calk_2$ is an \'etale equivalence relation on $\partial \calr_\infty$. Now, since 
$\partial \calr_\infty$ is compact and every $\calk_2$-class has two points, $\calk_2$ is also compact. To conclude, we show that $\calk_2$ is transverse to 
$\hat{\calr}_\infty \! \mid_{\partial \calr_\infty}$. It is clear that the intersection of $\hat{\calr}_\infty \! \mid_{\partial \calr_\infty}$ and $\calk_2 $ is reduced to the diagonal set $\Delta_{\partial \calr_\infty}$.
Moreover, for each element $((\calt_1,\calt_2),(\calt_2,\calt^\pitchfork_2)) \in \hat{\calr}_\infty \! \mid_{\partial \calr_\infty} \!\! \ast \, \calk_2$, there are local transformations $\varphi : A \cap \partial \calr_\infty \to B \cap \partial \calr_\infty$ of  $\hat{\calr}_\infty \! \mid_{\partial \calr_\infty}$
and  $\varphi^\pitchfork : A^\pitchfork \cap \partial \calr_\infty \to B^\pitchfork \cap \partial \calr_\infty$ of $\calk_2$ such that $\calt_2 = \varphi(\calt_1)$ and 
$\calt^\pitchfork_2 = \varphi^\pitchfork(\calt_2)$ are well-defined. We can assume $A^\pitchfork =A$. By denoting $\calt^\pitchfork_1 = \varphi^\pitchfork (\calt_1)$, we have: 
$$
\Phi : ((\calt_1,\calt_2),(\calt_2,\calt^\pitchfork_2)) \in \hat{\calr}_\infty \! \mid_{\partial \calr_\infty} \!\!  \ast \, \calk_2 
\mapsto 
((\calt_1,\calt^\pitchfork_1),(\calt^\pitchfork_1,\calt^\pitchfork_2)) \in \calk_2 \ast \hat{\calr}_\infty \! \mid_{\partial \calr_\infty}
$$
becomes a homeomorphim between the bisections of $\hat{\calr}_\infty \! \mid_{\partial \calr_\infty} \!\!  \ast \, \calk_2$ and $\calk_2 \ast \hat{\calr}_\infty \! \mid_{\partial \calr_\infty}$ defined by $\varphi$ and $\varphi^\pitchfork$, which extends naturally to a global topological isomorphism. 
\end{proof}

By applying the Absorption Theorem of \cite{GMPS2}, see Theorem~\ref{thabsorption}, we obtain the following result: 

\begin{proposition} \label{firstabsorption} The inductive limit 
$\hat{\calr} _\infty \vee \calk_2 = \varinjlim \hat{\calr} _n \vee \calk_2$ is a minimal $AF$ equivalence relation OE to the 
open $AF$ equivalence subrelation $\hat{\calr} _\infty$ of $\calr _\infty$. \qed
\end{proposition}

\subsection{Constructing $\calk_3$}
When we replace $\hat{\calr} _\infty$ with $\hat{\calr} _\infty \vee \calk_2$, all the elements of $\partial_2 \calr _\infty$ are absorbed, but some $\calr$-classes still split into several $\hat{\calr}_\infty$-equivalence classes because double-colored base points in $Y_2$ are never absorbed using $\calk_2$. 
%We could absorb a portion of these double-colored base points by restriction to the closed subset of $\partial \calr_\infty$ that we obtain taking the saturation of $\calh_{3,m}$ by the free $\mathbb{Z}_3$-action $\Psi_3$ described in Remark~\ref{remarkaction3}. Namely,  we can construct a CEER on this \'etale and thin closed set which is transverse to the restriction of $\hat{\calr} _\infty \vee \calk_2$. Using again the 
%Absorption Theorem of \cite{GMPS2}, we can deduce that  $\hat{\calr} _\infty \vee \calk_2  \vee \calk_{3,1} \vee \dots \vee \calk_{3,m}$ is a minimal $AF$ equivalence relation OE to $\hat{\calr} _\infty$. But this procedure makes it difficult to replace $\calk_{3,1} \vee \dots \vee \calk_{3,m}$ with its direct limit $ \varinjlim \calk_{3,1} \vee \dots \vee \calk_{3,m}$. In fact, the idea to construct $\calk_3$ will be very similar to that used in \cite{GMPS2}. 
From Proposition~\ref{semi-continuous}, the set $\{ \calt \in \partial \calr_\infty | D(\calt) \geq 3 \}$ is a closed subset of $\cpartial \calr_\infty$, which contains the fundamental domains $\check{\calh}_3$ and $\check{\calh}_4$ constructed in Propositions~\ref{H3},~\ref{H40}~and~\ref{H4k}. From Propositions~\ref{H40}~and~\ref{H3closed}, we know that $\check{\calh}_{4,0}$ and $\check{\calh}_3$ are closed subset of $\cpartial \calr_\infty$. On the other hand, when we construct $\check{\calh}_4$ in the proof of Proposition~\ref{H4k}, we can choose the configuration of the roots that belong to $\check{\calh}_4$. For example, belonging to a vertical semi-ribbon pointing upward or a horizontal semi-ribbon pointing to the right. This allows us to realize the union of $\check{\calh}_3$ and $\check{\calh}_4$ as a closed subset $\check{Y}_3$ of $\cpartial \calr_\infty$. Indeed, if  $\calt_n$ is a sequence in $\check{Y}_3$ that converges to  $\calt$, then $\calt$ belongs to the set $\{ \calt \in \partial \calr_\infty | D(\calt) \geq 3 \} \supset \check{\calh}_3 \cup \check{\calh}_4$ according to Proposition~\ref{semi-continuous}. Repeating the proof of Proposition~\ref{H3closed}, we deduce that  $\calt \in \check{Y}_3$. 
\medskip 

Now, combining Remark~\ref{remarkfd} with the discussion preceding Proposition~\ref{minimal'}~and~\ref{R'etale}, we can derive another closed set $Y_3 \subset Y_2 = \partial \calr_\infty$ in the following way. Each element in $ \check{Y}_3$ determine four double-colored base points in $X$ providing four different closed fundamental domains for $\calr$ and $\hat{\calr} _\infty \vee \calk_2$ . If the root belongs to $\check{\calh}_3$, then two base points are 
$\calr _\infty$-related, whereas the two other base points are not $\calr _\infty$-related between them. But recall that each base point is doubled and encoded by two color codes in $\calc$ at each stage of the inflation process. These color codes are equal for each of the two base points that are $\calr _\infty$-related, see Figure~\ref{calr'n.b}. For each root in $\check{\calh}_{4,\ell}$ with $\ell > 0$, the local situation around the root is the same than in the case above. Finally, for each root in $\check{\calh}_{4,0}$, the four base points, which are never $\calr_\infty$-related, are doubled and colored at each stage of the inflation process. Thus, we shall denote by $Y_3$ the union of two of these four closed fundamental domains whose elements are double-encoded by two different colors in $\calc$. This is a closed subset $Y_3$ of $Y_2 = \partial \calr_\infty$ verifying:

\begin{proposition} \label{Y3etale}
The set $Y_3$ is $\hat{\calr} _\infty \vee \calk_2$-thin and $\hat{\calr} _\infty \vee \calk_2$-\'etale.
\end{proposition} 

\begin{proof}
Firstly, the $\hat{\calr} _\infty \vee \calk_2$-thinness of $Y_3$ is an easy consequence of Proposition~\ref{R'thin}. On the other hand, since $Y_3$ is the disjoint union of two fundamental domains for $\hat{\calr} _\infty \vee \calk_2$, it is clear that $\hat{\calr} _\infty \vee \calk_2$ is \'etale on $Y_3$. 
\end{proof}

Now, we have an involutive homeomorphism from $Y_3$ onto itself changing the color code of each element of $Y_3$ and respecting each of the two fundamental domains in $Y_3$. This generates a CEER  $\calk_3$ on $Y_3$ which is not transverse to $\hat{\calr} _\infty \vee \calk_2 |_{Y_3}$. But restricting this involutive homeomorphism to each fundamental domain and extending trivially the corresponding CEERs to $Y_3$, we obtain two CEERs on $Y_3$ which are transverse to $\hat{\calr} _\infty \vee \calk_2 |_{Y_3}$ and whose union generates $\calk_3$. Using twice the Absorption Theorem of \cite{GMPS2}, see Theorem~\ref{thabsorption}, we have: 

\begin{proposition} \label{secondabsorption}
 The equivalence relation 
$\hat{\calr} _\infty \vee \calk_2 \vee \calk_3$ is a minimal $AF$ equivalence relation OE to $\hat{\calr} _\infty$. \qed
\end{proposition}

\subsection{Constructing $\calk_4$} Now, when we replace $\hat{\calr} _\infty \vee \calk_2$ by $\hat{\calr} _\infty \vee \calk_2 \vee \calk_3$, 
all the elements of $\partial_3 \calr _\infty$ are also absorbed, but there is a copy of $\check{\calh}_4$ that remains unabsorbed. In the final step of the absorption procedure, we shall consider again the union $Y_4$ of two copies of the fundamental domain $\calh_4$ constructed in Remark~\ref{remarkfd}. We choose these domains in such way that each base point around the second root is doubled by attaching two different color codes in $\calc$. We still have an involutive homeomorphism from $Y_4$ onto itself changing the color code of each element of $Y_4$. This generates a CEER $\calk_4$ on $Y_4$ which is not transverse to $\hat{\calr} _\infty \vee \calk_2 \vee \calk_3$. But $\calk_4$ is generated by the union of two CEERs on $Y_4$ which are transverse to $\hat{\calr} _\infty \vee \calk_2 \vee \calk_3$. In this case, we need to restrict the involutive homeomorphism to a complete copy of $\calh_4$  in the first case, and to a copy of $\calh_4 - \calh_{4,0}$ in the second case. Next, we extend trivially the corresponding CEERs to $Y_4$. In order to apply the Absorption Theorem of \cite{GMPS2}, we need to restrict $\calk_4$ to the increasing sequence of closed subsets $Y_{4,0} \cup Y_{4,1} \cup \dots \cup Y_{4,\ell}$ of $Y_4$ (all of whose elements are encoded by two different colors) derived from the fundamental domains $\check{\calh}_{4,0} \cup \check{\calh}_{4,1} \dots \cup \check{\calh}_{4,\ell}$ as before:  

\begin{proposition} \label{thirdabsorption} 
For all $\ell \geq 0$, the equivalence relation
$$\hat{\calr} _\infty \vee \calk_2  \vee \calk_3 \vee
 \calk_4 |_{\textstyle Y_{4,0} \cup Y_{4,1} \cup \dots \cup Y_{4,\ell}}$$
 is a minimal $AF$ equivalence relation OE to $\hat{\calr} _\infty$. \qed
\end{proposition}

\noindent
%Finally, since $\calr = \hat{\calr} _\infty \vee \calk_2 \vee \calk_3 \vee \calk_4$, Theorem~\ref{affabilitytheorem} follows from Propositions \ref{firstabsorption}, \ref{secondabsorption} and \ref{thirdabsorption}. Note, however, that this last step does not allow us to assert that $\calr$ is OE to  $\hat{\calr} _\infty$.
Finally, since $\calr = \hat{\calr} _\infty \vee \calk_2 \vee \calk_3 \vee \calk_4$, Theorem~\ref{affabilitytheorem} follows 
as $\calr$ can be described as the direct limit 

\[
  \calr = \varinjlim \hat{\calr} _\infty \vee \calk_2  \vee \calk_3 \vee  \calk_4 |_{\textstyle Y_{4,0} \cup Y_{4,1} \cup \dots \cup Y_{4,\ell}}
\]
which should be affable. Note, however, that this last step does not allow us to assert that $\calr$ is OE to  $\hat{\calr} _\infty$.

\end{document}